\documentclass[a4paper,11pt,DIV=11,% decrease borders by 2 (std 10 for 11pt)
abstract=on% return to old style centered abstract
]{scrartcl}
\usepackage{mathtools}
\mathtoolsset{showonlyrefs}
\usepackage{amsmath, amsthm, amssymb}
\usepackage{fontenc}
\usepackage[utf8]{inputenc}
\usepackage[english]{babel}
\usepackage{graphicx}
\usepackage{subcaption,xargs}
\usepackage{algorithm, booktabs}
\usepackage[noend]{algpseudocode}
\usepackage{enumitem,listings}%,microtype}
\setlist[itemize]{label=\textbullet}
\usepackage{multirow}
\usepackage{hyperref}
\DeclareCaptionLabelSeparator{periodspace}{.\ }
\newtheorem{theorem}{Theorem}[section]

\newtheorem{proposition}[theorem]{Proposition}
\newtheorem{lemma}[theorem]{Lemma}
\newtheorem{remark}[theorem]{Remark}
\newtheorem{example}[theorem]{Example}

\setkomafont{sectioning}{\rmfamily\bfseries}
\setkomafont{title}{\rmfamily}
\setkomafont{descriptionlabel}{\rmfamily\bfseries}

\newcommand{\C}{\mathbb{C}}

\newcommand{\RR}{\mathbb{R}}

\newcommand{\NN}{\mathbb{N}}

\newcommand{\dist}{{\mathrm{d}}}
\newcommand{\HH}{\mathcal{H}}
\newcommand{\Hn}{\mathbb{H}}
\newcommand{\SPD}{\mathcal P}

\makeatletter
\newcommandx{\abs}[2][1=\@empty]{#1\lvert #2 #1\rvert}
\newcommandx{\norm}[3][1=\@empty,3=\@empty]{#1\lVert #2 #1\rVert_{#3}}
\makeatletter
\newcommand{\vect}[1]{\mathbf{#1}}

\newcommand{\tT}{\mathrm{T}}
\DeclareMathOperator*{\argmin}{arg\,min}
\DeclareMathOperator{\prox}{prox}

\DeclareMathOperator{\dom}{dom}
\DeclareMathOperator{\ri}{ri}
\DeclareMathOperator{\interior}{int}

\DeclareMathOperator{\arcosh}{arcosh}
\DeclareMathOperator{\arsinh}{arsinh}

\DeclareMathOperator{\TV}{TV}

\DeclareMathOperator{\Log}{Log}
\DeclareMathOperator{\Exp}{Exp}

\hypersetup{pdfauthor={Ronny Bergmann, Johannes Persch, Gabriele Steidl},
	pdftitle={A Parallel Douglas Rachford Algorithm
for Minimizing ROF-like Functionals 
on Images with Values in Symmetric Hadamard Manifolds},
	pdfsubject={revised manuscript},%
	pdfcreator = {pdflatex and TextMate},%
	pdfkeywords={},
	plainpages=false, pdfstartview=FitH, pdfview=FitH, pdfpagemode=UseOutlines,%
	bookmarksnumbered=true,bookmarksopen=false,bookmarksopenlevel=0,%
	colorlinks=true,linkcolor=black,citecolor=black,urlcolor=black%
}
\begin{document}
\title{
A Parallel Douglas Rachford Algorithm
for Minimizing ROF-like Functionals 
on Images with Values in Symmetric Hadamard Manifolds}
\subtitle{Extended Version}
\date{\today}
\author{Ronny Bergmann\footnote{Department of Mathematics,
Technische Universität Kaiserslautern,
Paul-Ehrlich-Str.~31, 67663 Kaiserslautern, Germany,
$\{$bergmann, persch, steidl$\}$@mathematik.uni-kl.de.}
\and Johannes Persch\footnotemark[1] \and Gabriele Steidl\footnotemark[1]}
\maketitle%
\begin{abstract}
	\noindent\small 
	We are interested in restoring images having values in a symmetric Hadamard manifold
	by minimizing a functional with a quadratic data term and a total variation like regularizing term.
	To solve the convex minimization problem, we extend the Douglas-Rachford algorithm 
	and its parallel version to symmetric Hadamard manifolds.
	The core of the Douglas-Rachford algorithm are reflections of the functions involved
	in the functional to be minimized.
	In the Euclidean setting the reflections of convex lower semicontinuous functions
	are nonexpansive. 
	As a consequence, convergence results for Krasnoselski-Mann iterations  imply the convergence of the 
	Douglas-Rachford algorithm.
	Unfortunately, this general results does not carry over to Hadamard manifolds, where
	proper convex lower semicontinuous functions can have expansive reflections. 
	However, splitting our restoration functional in an appropriate way, we have 
	only to deal with special functions namely, several distance-like functions
	and an indicator functions of a special convex sets.
	We prove that the reflections of certain distance-like functions on Hadamard manifolds are nonexpansive
	which is an interesting result on its own.
	Furthermore, the reflection of the involved indicator function is nonexpansive
	on Hadamard manifolds with constant curvature so that the Douglas-Rachford algorithm converges here.
		
	Several numerical examples demonstrate the advantageous performance of 
	the suggested algorithm compared to other existing methods
	as the cyclic proximal point algorithm or half-quadratic minimization.
	Numerical convergence is also observed in our experiments 
	on the Hadamard manifold of symmetric positive definite matrices  
	with the affine invariant metric which does not have a constant curvature.	
\end{abstract}

%----------------------------------------------
\section{Introduction} \label{sec:intro}
%----------------------------------------------
%
In the original paper~\cite{DR56}, the Douglas-Rachford (DR) algorithm was
proposed for the numerical solution of a partial differential equation, 
i.e., for solving systems of linear equations. 
It was generalized for finding a zero of the sum of two maximal monotone operators 
by Lions and Mercier \cite{LM79}, see also Passty's paper \cite{Pa79}.
Eckstein and Bertsekas  \cite{EB92} examined the algorithm with under/over-relaxation and inexact inner evaluations.
Gabay \cite{Ga83} considered problems of a special structure and showed that the DR algorithm applied to
the dual problem results in the alternating direction method of multipliers (ADMM) introduced in \cite{GM76,Gl84}.
For relations between the DR algorithm and the ADMM we also refer to \cite{gl14,Se09}.
In \cite{YY14} it is shown that the ADMM is in some sense self-dual, i.e.,
it is not only equivalent to the DR algorithm applied to the dual problem, but also to the primal one.
Recently, these algorithms were successfully applied in image processing mainly for two reasons:
the functionals to minimize allow for simple proximal mappings within the method,
and it turned out that the algorithms are highly parallelizable, see, e.g., \cite{CP08}.
Therefore the algorithms became one of the most popular ones in variational methods for image processing.
Continued interest in DR iterations is also due to its excellent, but still 
myserious performance  on various non-convex problems, see, e.g.,
\cite{BCL2002,BCL2003,BS2011,ERT2007,GE2008,HL2012} and for recent progress on the convergence of  ADMM methods for special
non-convex problems \cite{HLR2014,LP2014,MWRF2014,WCX2015,WYZ15,XYWZ2012}.

In this paper, we derive a parallel DR algorithm to minimize functionals on finite dimensional, symmetric Hadamard manifolds.
The main ingredients of the algorithm are reflections of the functions involved in the functional we want to minimize.
If these reflections are nonexpansive, then it follows from general convergence results of  
Krasnoselski-Mann iterations in CAT(0) spaces \cite{K2013} that the sequence produced by the algorithm converges.
Unfortunately, the well-known result in the Euclidean setting that reflections of
convex lower semicontinuous functions are nonexpansive does  not carry over to
the Hadamard manifold setting, see \cite{BH1999,FL2013}. Up to now it was only proved that
indicator functions of closed convex sets in Hadamard manifolds with constant curvature 
possess nonexpansive reflections \cite{FL2013}.
In this paper, we prove that certain distance-like functions have nonexpansive reflections on
general symmetric Hadamard manifolds. Such distance-like functions appear for example in the
manifold-valued counterpart of the Rudin-Osher-Fatemi (ROF) model~\cite{ROF92}. 

The ROF model is the most popular variational model for image restoration, 
in particular for image denoising.
For real-valued images, its discrete, anisotropic penalized form is given by
\begin{equation} \label{rof_real}
\begin{split}
	D(u;f) + \alpha \TV(u) &=\tfrac12 \lVert f - u \rVert_2^2 + \alpha \lVert\nabla u\rVert_1\\
		&= \tfrac12\sum_{i,j} (f_{i,j} - u_{i,j})^2 + \alpha\sum_{i,j} \bigl(  \abs{u_{i+1,j} - u_{i,j}} + \abs{u_{i,j+1} - u_{i,j}} \bigr), 
\end{split}
\end{equation}
where~$f = (f_{i,j}) \in \mathbb R^{N,M}$ 
is an initial corrupted image and $\nabla$ denotes
the discrete gradient operator usually 
consisting of first order forward differences in vertical and horizontal directions.
The first term is the \emph{data fidelity term} \(\mathcal D(u; f)\) measuring
similarity between \(u\) and the given data \(f\). The second term
\(\TV(u)\) is the total variation (TV) type \emph{regularizer} posing a small value of the
first order differences in \(u\). 
The regularization parameter $\alpha > 0$ steers the relation between both terms.
The popularity of the model arises from the fact that its minimizer is
a smoothed image that preserves important features such as edges.
There is a close relation of the ROF model to PDE and wavelet approaches, see \cite{SWBMW04}.

In various applications in image processing and computer vision the functions
of interest take values in a Riemannian manifold. 
One example is diffusion tensor imaging where the data is given on the Hadamard
manifold of positive definite matrices; see,
e.g.,~\cite{basser1994mr,BWFW04,BDFW2007,CTF2002,pennec2006riemannian,SSPB07,WFWBB06,WFBW03}.
In the following we are interested in generalizations of ROF-like functionals 
to manifold-valued settings, more precisely to data having values in symmetric Hadamard manifolds.
In~\cite{GM06,GM07}, the notion of the total variation of functions
having their values on a manifold was investigated based on
the theory of Cartesian currents.
The first work which applies a TV approach of circle-valued data
for image processing tasks is~\cite{SC11}. An
algorithm for TV regularized minimization problems on Riemannian
manifolds was proposed  in~\cite{LSKC13}. There, 
the problem is reformulated as a multilabel optimization problem
which is approached using convex relaxation techniques.
Another approach to TV minimization for manifold-valued data which
employs cyclic and parallel proximal point algorithms
and does not require labeling and relaxation techniques, was given
in~\cite{WDS2014}. A method which circumvents the direct work with manifold-valued data by
embedding the matrix manifold in the appropriate Euclidean space and applying
a back projection to the manifold was suggested in~\cite{RTKB14}. 
TV-like functionals on manifolds with higher order differences were handled in \cite{BBSW2015,BLSW14,BW15a}.
Finally we mention the relation to wavelet-type multiscale
transforms which were handled, e.g., in~\cite{GW2009,GW2012,RDSDS2005,WYG2007}.

We will apply the parallel DR algorithm to minimize the ROF-like functional
for images having values in a symmetric Hadamard manifolds.
To this end we will split the functional in an appropriate way and 
show the convergence of the algorithm by examining the reflections of the involved distance-like functions.
In the numerical part we show the very good performance of the proposed parallel DR algorithm
for various symmetric Hadamard manifolds with and without constant curvature.
In particular, we compare the algorithm with other algorithms
existing in the literature, namely the cyclic proximal point algorithm \cite{WDS2014}
and a half-quadratic minimization method applied to a smoothed version of the ROF-like functional \cite{BCHPS15}.

%-------------------------------------------------------------------------------
The outline of the paper is as follows:
We start by recalling the DR algorithm and its parallel version in the Euclidean setting
in Section~\ref{sec:hilbert}. The generalization to symmetric Hadamard manifolds will follow the same path.
In Section~\ref{sec:notation} we provide the notation and preliminaries in Hadamard manifolds
which are required to understand our subsequent findings. The parallel DR algorithm on symmetric Hadamard manifolds
is outlined in Section~\ref{sec:hadamard}. We prove that the sequence produced by the algorithm 
converges to a minimizer of the functional.
In Section~\ref{sec:rof} we show how the parallel DR algorithm can be applied to minimize
a ROF-like functional which can be used for restoring images with values in symmetric Hadamard manifolds.
Convergence of the algorithm is ensured if the reflections of the functions appearing in the functional 
are nonexpansive. For the ROF-like functional we have, due to an appropriate splitting, only to consider distance-like functions and an indicator function of a convex set.
Section~\ref{subsec:dist} contains the interesting result that reflections of certain distance-like functions on
symmetric Hadamard manifolds are nonexpansive.
In Section \ref{subsec:proj} we will see that indicator functions of closed convex sets have nonexpansive reflections 
on manifolds with constant curvature.
Numerical examples are demonstrated in Section \ref{sec:numerics}. These include manifolds such as
the hyperbolic model space and the space of symmetric positive definite matrices. Comparisons with
 other algorithms to minimize an ROF-like functional on Hadamard manifolds are given.
Conclusions are drawn in Section \ref{sec:conclusions}.
The appendix provides material on symmetric positive definite matrices and hyperbolic spaces
which are necessary for the implementation of the generalized parallel DR algorithm.
%
%----------------------------------------------
\section{Parallel DR Algorithm on Euclidean Spaces} \label{sec:hilbert}
%----------------------------------------------
%
We start by recalling the DR algorithm and its parallel form on Euclidean spaces.
We will follow the same path for data in a Hadamard manifold in Section~\ref{sec:hadamard}.
The main ingredients of the DR algorithm are proximal mappings and reflections.

For $\eta>0$ and a proper convex lower semicontinuous (lsc) function
$\varphi\colon \mathbb R^n \rightarrow (-\infty,+\infty]$, 
the \emph{proximal mapping} $\prox_{\eta \varphi}$ reads
\begin{equation} \label{prox_hilbert}
	\prox_{\eta \varphi}(x)
	\coloneqq \argmin_{y \in \mathbb R^n} \Big\{ \frac{1}{2} \lVert x-y\rVert_2^2 + \eta\varphi (y) \Big\},
\end{equation}
see~\cite{Mor62}. It is well-defined and unique.
The \emph{reflection} $R_p\colon \mathbb R^n \rightarrow \mathbb R^n$  at a point $p \in \mathbb R^n$ is given by
\begin{equation} \label{reflection_R}
	R_p (x) = 2p - x.
\end{equation}
Further, let
${\mathcal R}_\varphi\colon \mathbb R^n \rightarrow \mathbb R^n$
denote the reflection operator at $\prox_{ \varphi}$, i.e.,
\begin{equation} \label{reflection_R_1}
{\mathcal R}_\varphi (x) = 2\prox_{ \varphi}(x) - x.
\end{equation}
We call this operator \emph{reflection of the function}  $\varphi$.

Given two proper convex lsc functions
$\varphi,\psi\colon\mathbb R^n \rightarrow (-\infty,+\infty]$ the DR algorithm 
aims to solve
\begin{equation} \label{drs_split}
 \argmin_{x \in \mathbb R^n} \bigl\{ \varphi(x) + \psi(x) \bigr\}
\end{equation}
by the steps summarized in Algorithm~\ref{alg:DR_real}.
%
%---------------------------------------------------
\begin{algorithm}[tbp]
	\caption[]{DR Algorithm for Real-Valued Data} 
	\label{alg:DR_real}
	\begin{algorithmic}
		\State \textbf{Input:}
		$t^{(0)}\in \mathbb R^n$,
		$\lambda_r \in[0,1]$
		with $\sum_{r \in \mathbb N} \lambda_r(1-\lambda_r) = + \infty$, 
		$\eta > 0$
		\State r = 0;
		\Repeat
		\State $t^{(r+1)}
		= \bigl(
			(1-\lambda_r)\operatorname{Id}
			+ \lambda_r {\mathcal R}_{\eta \varphi} {\mathcal R}_{\eta \psi}
			\bigr) \bigl(t^{(r)}\bigr)$;
		\State $r\rightarrow r+1$;
		\Until a stopping criterion is reached
	\end{algorithmic}
\end{algorithm}
%---------------------------------------------------
%
It is known that the DR algorithm
converges for any proper convex lsc functions $\varphi,\psi$
under mild assumptions.
More precisely, we have the following theorem,
see~\cite{LM79} or~\cite[Theorem 27.4]{BC11}.
%
%---------------------------------------------------
\begin{theorem} \label{conv_drs_real}
Let $\varphi,\psi\colon \mathbb R^n \rightarrow (-\infty,+\infty]$ be
proper convex lsc functions such that
$\ri(\dom \varphi) \cap \ri(\dom \psi) \not = \emptyset$.
Assume that a solution of~\eqref{drs_split} exists. 
Let $\{ \lambda_r \}_{r \in \mathbb N}$
fulfill $\sum_{r \in \mathbb N} \lambda_r(1-\lambda_r) = + \infty$ and $\eta >0$.
Then the sequence $\bigl\{t^{(r)} \bigr\}_{r \in \mathbb N}$ generated by
Algorithm~\ref{alg:DR_real} converges for any starting point $t^{(0)}$
to a point $\hat t$, and 
\begin{equation} \label{doch**}
	\hat x\coloneqq \prox_{\eta \psi} (\hat t)
\end{equation}
is a solution of~\eqref{drs_split}.
%If we replace $\mathbb R^n$ by a Hilbert space which is not finite dimensional  only weak convergence of the sequence $\bigl\{t^{(r)}\bigr \}_{r \in \mathbb N}$ is ensured.
\end{theorem}
%---------------------------------------------------
%
The DR algorithm can be considered as a special case of the
\emph{Krasnoselski-Mann iteration}
\begin{equation} \label{it:km_real}
t^{(r+1)} = \bigl( (1-\lambda_r)\operatorname{Id}+ \lambda_r T\bigr)(t^{(r)})
\end{equation}
with $T\coloneqq {\mathcal R}_{\eta \varphi} {\mathcal R}_{\eta \psi}$.
%Clearly,  $T$ and $\frac12(\operatorname{Id}+T)$ have the same fixed point set.
Since the concatenation and convex combination of nonexpansive operators
is again nonexpansive, the proof is just based on the fact that
the reflection operator ${\mathcal R}_{\eta \varphi}$ of a proper convex lsc function $\varphi$ 
is nonexpansive. The latter will not remain true in the Hadamard manifold setting. 

\begin{remark}
The definition of the DR algorithm is dependent on the order 
of the operators ${\mathcal R}_{\eta \varphi}$ and ${\mathcal R}_{\eta \psi}$
although problem~\eqref{drs_split} itself is not. 
The order of the operators in the DR iterations was examined in~{\rm \cite{BM2015}}.
The authors showed that ${\mathcal R}_{\eta \varphi}$ is an isometric bijection from the fixed point set of 
$\frac12(\operatorname{Id}+  {\mathcal R}_{\eta \psi}{\mathcal R}_{\eta \varphi})$
to that of $\frac12(\operatorname{Id}+  {\mathcal R}_{\eta \varphi}{\mathcal R}_{\eta \psi})$ 
with inverse ${\mathcal R}_{\eta \psi}$.
For the effect of different orders of the operators in the ADMM algorithm see~{\rm \cite{YY2015}}.
\end{remark}

In this paper, we are interested in multiple summands.
We consider
\begin{align} \label{eq:many}
\argmin_{x\in \mathbb R^n} \Bigl\{ \sum_{k=1}^K \varphi_k(x)\Bigr\},
\end{align}
where $\varphi_k\colon \mathbb R^n \rightarrow (-\infty,+\infty]$, $k=1,\ldots,K$, are proper convex lsc functions with \\
$\bigcap_{k=1}^K\ri(\dom \varphi_k)\neq\emptyset$.
The problem can be rewritten in the form~\eqref{drs_split} with only two summands,
namely
\begin{equation}\label{eq:prod_problem}
\argmin_{{\vect{x}}\in \mathbb R^{nK} }
	\bigl\{\Phi(\vect{x})+\iota_{\textsf{D}}(\vect{x}) \bigr\},
\end{equation}
where 
$\Phi (\vect{x})\coloneqq\sum_{k=1}^K \varphi_k(x_k)$, 
$\vect{x}\coloneqq (x_k)_{k=1}^K$, $x_k \in \mathbb R^n$, $k=1,\ldots,K$, and
\[
	\iota_{\textsf{D}} (\vect{x})
	\coloneqq
	\begin{cases}
		0 & \mbox{ if } \vect{x} \in \textsf{D},\\
		\infty & \mbox{ otherwise,}
	\end{cases}
\]
is the indicator function of
\[
\textsf{D}\coloneqq \bigl\{ \vect{x} \in \mathbb R^{nK} \colon x_1 = \ldots = x_K\in \mathbb R^n \bigr\}.
\]
Since ${\textsf{D}}$ is a nonempty closed convex set, its indicator function is
proper convex and lsc.
Further, we have 
\begin{equation}\label{prox_iota_D}
 \prox_{\iota_{\textsf{D}}}(\vect{x}) = \Pi_{\textsf{D}} (\vect{x})
 =
% \vecOne_n \otimes \frac{1}{n} \sum_{k=1}^n x_k.
\Big(
	%\underbrace{
	\frac{1}{K} \sum_{k=1}^K x_k, \ldots, \frac{1}{K} \sum_{k=1}^K x_k
	%}_{K \; \text{times}}
 \Big)
 \in \mathbb R^{nK}
\end{equation}
%where $\vecOne_n$ is the vector having $n$ entries 1 and $\otimes$ denotes the Kronecker product. 
By $\Pi_{\textsf{D}}$ we denote the orthogonal projection onto $\textsf{D}$.
If~\eqref{eq:prod_problem} has a solution, then we can apply the
DR Algorithm~\ref{alg:DR_real} to the special setting
in~\eqref{eq:prod_problem} and obtain Algorithm~\ref{alg:DR_par_real}.
We call this algorithm {\it parallel DR algorithm}. 
In the literature, it is also known as product version of the DR algorithm~\cite{BT2014}, 
parallel proximal point algorithm~\cite{CP08}
or, in a slightly different version, as parallel splitting algorithm~\cite[Proposition 27.8]{BC11}.
For a stochastic version of the algorithm we refer to~\cite{CP2014}.

%
%---------------------------------------------------
\begin{algorithm}[tbp]
	\caption[]{Parallel DR Algorithm for Real-Valued Data}
	\label{alg:DR_par_real}
	\begin{algorithmic}
		\State \textbf{Input:}  $\vect{t}^{(0)}\in \mathbb R^{nK}$, $\lambda_r \in[0,1]$ with $\sum_{r \in \mathbb N} \lambda_r(1-\lambda_r) = + \infty$, 
		$\eta > 0$
		\State r = 0;
		\Repeat
		\State $\vect{t}^{(r+1)} = \bigl((1-\lambda_r)\operatorname{Id}+ \lambda_r {\mathcal R}_{\eta \Phi} {\mathcal R}_{\iota_{\textsf{D}}}\bigr) \bigl(\vect{t}^{(r)}\bigr)$;
		\State $r\rightarrow r+1$;
		\Until a stopping criterion is reached
	\end{algorithmic}
\end{algorithm}
%---------------------------------------------------
%

By Theorem~\ref{conv_drs_real}, the convergence of the iterates of Algorithm~\ref{alg:DR_par_real} to a point
${\vect{\hat t}} = (\hat t_1,\ldots, \hat t_K)^\tT \in \mathbb R^{nK}$
is ensured under the assumptions of Theorem~\ref{conv_drs_real}. 
Further, by~\eqref{doch**} and~\eqref{prox_iota_D}, we obtain that 
\[
\hat x \coloneqq \frac{1}{K} \sum_{k=1}^K \hat t_k 
\]
is a solution of~\eqref{eq:many} resp.~\eqref{eq:prod_problem}.

Finally, we want to mention the so-called cyclic DR algorithm.

%
%------------------------------------------------------
\begin{remark} (Cyclic DR Algorithm) \label{rem:cyclic_drs}
	We consider the cyclic DR algorithm for~\eqref{eq:many}.
	Let $S_k := {\mathcal R}_{\eta \varphi_{k+1}}{\mathcal R}_{\eta \varphi_k}$,
	$k=1,\ldots,K-1$, and 
	$S_K := {\mathcal R}_{\eta \varphi_{1}}{\mathcal R}_{\eta \varphi_K}$.	
	Starting with $t^{(0)} \in \mathbb R^n$
	we compute
	\[
		t^{(r+1)} = T_{[\varphi_1, \ldots,\varphi_K]} t^{(r)},
	\]
	where
	\[
		T_{[\varphi_1, \ldots,\varphi_K]} \coloneqq 
		\tfrac12 \bigl(\operatorname{Id}
			+ S_{K} \bigr)\circ
		\tfrac12 \bigl(\operatorname{Id} 
			+ S_{K-1} \bigr) \circ
		\ldots\circ
		\tfrac12 \bigl(\operatorname{Id}
			+ S_{2} \bigr)\circ
		\tfrac12 \bigl(\operatorname{Id} 
			+ S_{1} \bigr).
	\]
	Applying the operator $T_{[\varphi_1, \ldots,\varphi_K]}$ can be written
	again in the form of a Krasnoselski-Mann iteration~\eqref{it:km_real}
	with operator
	$
	\mu {\operatorname{Id}} + (1-\mu) T
	$
	where $\mu = (\frac{1}{2})^K$ and $T$ is the convex combination
	\[
		\tfrac{1}{2^K-1} \big(S_1 + \ldots
			+ S_K + S_2\circ S_1 + \ldots + S_K\circ S_{K-1} + \ldots + S_K\circ \ldots\circ S_1 \big).
	\]
	Thus it can be shown that the sequence $\bigl\{t^{(r)}\bigr\}_{r \in \mathbb N}$ converges
	to some $\hat t$. However, only for indicator functions $\varphi_k\coloneqq \iota_{{\mathcal C}_k}$
	of closed, convex sets
	${\mathcal C}_k \not = \emptyset$ it was proved
	in~{\rm \cite{BT2014}} that
	$\hat x
		= \prox_{\iota_{{\mathcal C}_k}} (\hat t)
		= \Pi_{{\mathcal C}_k}(\hat t)$
		is a solution
	of~\eqref{eq:many}. In other words, the algorithms finds an element of 
	$\cap_{k=1}^K {\mathcal C}_k$.
  For non-indicator functions $\varphi_k$, there are to the best of our knowledge no similar
	results, i.e., the algorithm converges but the meaning of the limit is not clear.
	The same holds true for an averaged version of
	the DR algorithm, see~{\rm \cite{BT2014}}.
\end{remark}
%------------------------------------------------------
%

%We want to generalize the parallel DR algorithm for data in Hadamard manifolds.
%First we have to introduce the basic notation.
%
%----------------------------------------------
\section{Preliminaries on  Hadamard Manifolds} \label{sec:notation}
%----------------------------------------------
%
A complete metric space $(\HH,d)$ is called a \emph{Hadamard space} if every two
points $x,y$  are connected by a geodesic and the following condition holds true
\begin{equation}\label{eq:reshet}
d(x,v)^2 + d(y,w)^2 \le d(x,w)^2 + d(y,v)^2 + 2 d(x,y)d(v,w),
\end{equation}
for any $x,y,v,w \in \HH$. 
Inequality~\eqref{eq:reshet} implies that Hadamard spaces have nonpositive
curvature~\cite{Alex51,Re68}. 
In this paper we restrict our attention to Hadamard spaces
which are at the same time finite dimensional Riemannian manifolds with geodesic distance~$d$.

Let $T_x\HH$ be the tangent space at $x$ on $\HH$. 
By $\exp_x\colon T_x\HH \rightarrow \HH$ we denote the exponential map at $x$. 
Then we have
$\exp_x \xi = \gamma_{x,\xi} (1) = y$,
where $\gamma_{x,\xi}:[0,1] \rightarrow \HH$ is the unique geodesic starting at $x$ in direction $\xi$,
i.e., $\gamma_{x,\xi}(0) = x$ and $\dot{\gamma}_{x,\xi}(0)=\xi$.
Likewise the geodesic starting at $x$ and reaching $y$ at time $t=1$ is denoted by $\gamma_{\overset{\frown}{x,y}}$.
Its inverse $\log_x\colon \HH \rightarrow T_x\HH$ is given by
$\log_x y = \xi$, where $\gamma_{x,\xi}(1) = y$.

Hadamard spaces have the nice feature that they resemble convexity properties from Euclidean spaces.
A \emph{set} ${\mathcal C} \subseteq \HH$ is \emph{convex}, if for any $x,y \in \HH$ the
geodesic $\gamma_{\overset{\frown}{x,y}}$ lies in $\mathcal C$.
The intersection of an arbitrary family of convex closed sets is itself a
convex closed set.
A \emph{function} $\varphi\colon \HH \rightarrow (-\infty,+\infty]$ is called \emph{convex} if
for any $x,y \in \HH$ the function $\varphi \circ \gamma_{\overset{\frown}{x,y}}$ is
convex, i.e., we have 
\[
	\varphi\bigl( \gamma_{\overset{\frown}{x,y}}(t) \bigr)
	\le t \varphi\bigl( \gamma_{\overset{\frown}{x,y}} (0) \bigr)
	+ (1-t)\varphi \bigl(\gamma_{\overset{\frown}{x,y}}(1)\bigr), \text{ for all } t \in [0,1].
\]
The function $\varphi$ is
\emph{strictly convex} if the strict inequality holds true for all $0 < t <1$, and
\emph{strongly convex}  with parameter $\kappa > 0$ if for any $x,y \in \HH$ and all $t \in [0,1]$ we have
\[
\varphi\bigl(\gamma_{\overset{\frown}{x,y}}(t) \bigr)
	\le t \varphi\bigl(\gamma_{\overset{\frown}{x,y}} (0) \bigr) 
	+ (1-t)\varphi \bigl(\gamma_{\overset{\frown}{x,y}}(1)\bigr)
	- \kappa t(1-t) d\bigl(\gamma_{\overset{\frown}{x,y}}(0),\gamma_{\overset{\frown}{x,y}}(1)\bigr).
\]
The distance in an Hadamard space fulfills the following properties: 
\begin{enumerate}[label={(D\arabic*)}]
\item\label{item:DandDSq-convex}
	$d\colon {\mathcal H} \times {\mathcal H} \rightarrow \mathbb R_{\ge 0}$ 
	and 
	$d^2\colon {\mathcal H} \times {\mathcal H} \rightarrow \mathbb R_{\ge 0}$
	are convex, and
\item\label{item:DSq-stronly-convex}
	$d^2(\cdot,y)\colon {\mathcal H} \rightarrow \mathbb R_{\ge 0}$
	is strongly convex with $\kappa = 1$.
\end{enumerate}
Concerning minimizers of convex functions the next theorem summarizes some basic facts
which can be found, e.g., in~\cite[Lemma 2.2.19]{B2014} and~\cite[Proposition 2.2.17]{B2014}.

%
%----------------------------------------------
\begin{theorem} \label{hadamardmin}
For proper convex lsc functions $\varphi\colon \HH \rightarrow (-\infty,+\infty]$ we have:
\begin{enumerate}[label={\textup{\roman*)}}]
	\item\label{item:pclsc-minimizer}
		If $\varphi(x) \rightarrow +\infty$
		whenever $d(x,x_0) \rightarrow +\infty$ for some $x_0 \in \HH$,
		then $\varphi$ has a minimizer.
	\item\label{item:psclsc-unique}
		If $\varphi$ is strongly convex,
		then there exists a unique minimizer of $\varphi$.
\end{enumerate}
\end{theorem}
%----------------------------------------------
%

The \emph{subdifferential} of $\varphi\colon \HH \rightarrow (-\infty,+\infty]$
at $x \in \dom \varphi$ is defined by 
\[
	\partial \varphi (x)
	\coloneqq
	\bigl\{
		\xi \in T_x\HH\colon \varphi(y) \ge \varphi(x) 
			+ \langle \xi,\dot \gamma_{\overset{\frown}{x,y}}(0) \rangle
			\text{ for all } y \in \dom \varphi
	\bigr\},
\]
see, e.g.,~\cite{LMWY2011} or~\cite{Udriste94} for finite functions $\varphi$.
For any $x \in \interior (\dom \varphi)$, the subdifferential is a nonempty convex and compact set in $T_x\HH$.
If the Riemannian gradient $\nabla \varphi(x)$ of $\varphi$ in $x \in \HH$ exists, 
then $\partial \varphi (x) = \bigl\{\nabla \varphi(x)\bigr\}$. 
Further we see from the definition that $x \in \HH$ is a global minimizer of $\varphi$ if and only if $0 \in \partial \varphi (x)$.
The following theorem was proved in~\cite{LMWY2011} for general finite dimensional Riemannian manifolds.
%
%-------------------------------
\begin{theorem} \label{th:subdiff_calculus}
Let $\varphi,\psi\colon \HH \rightarrow (-\infty,+\infty]$ be proper and convex. 
Let $x \in \interior (\dom \varphi) \cap \dom \psi$.
Then we have the subdifferential sum rule
\[
	\partial(\varphi + \psi)(x) = \partial \varphi (x) + \partial \psi (x).
\]
In particular, for a convex function $\varphi$ 
and a nonempty convex set ${\mathcal C}$ such that
${\mathcal C}\cap \dom \varphi \not = \emptyset$ is convex, it follows 
for $x \in \interior (\dom \varphi) \cap {\mathcal C}$ that
\[
	\partial(\varphi + \iota_{{\mathcal C}})
	= \partial \varphi (x) + N_{{\mathcal C}}(x),
\]
where the normal cone $N_{{\mathcal C}}(x)$ of ${\mathcal C}$
at $x \in {\mathcal C}$ is defined by
\[
	N_{{\mathcal C}} (x)
	\coloneqq \bigl\{
		\xi \in T_x \HH\colon\langle \xi,\log_xc \rangle \le 0
		\text{ for all } c \in \mathcal C
	\bigr\}.
\]
\end{theorem}
%-------------------------------------------------------------------------------
%

We will deal with the product space $\HH^n$ of a Hadamard manifold $\HH$
with distance
\begin{equation} \label{prod_dist}
d(x,y)\coloneqq \Big( \sum_{k=1}^n d(x_k,y_k)^2 \Big)^\frac12
\end{equation}
which is again a Hadamard manifold.

For $\eta > 0$ and a proper convex lsc function
$\varphi\colon\HH \rightarrow [-\infty,+\infty]$, the \emph{proximal mapping} 
\begin{equation} \label{eq:prox_mani}
\prox_{\eta \varphi }(x)
\coloneqq \argmin_{y \in {\HH^n}} \bigl\{ \tfrac{1}{2}  d (x,y) ^2
	+ \eta\varphi(y)\bigr\}
\end{equation}
exists and is uniquely determined, see~\cite{Jost,Mayer}. 

To introduce a DR algorithm we have to define reflections on manifolds.
A mapping $R_p\colon{\mathcal M} \rightarrow {\mathcal M}$ 
on a Riemannian manifold ${\mathcal M}$
is called \emph{geodesic reflection} at $p \in {\mathcal M}$, if
\begin{equation}
	R_p(p) = p, \quad \text{and} \quad D_p(R_p) = -I,
\end{equation}
where $D_p(R_p)$ denotes the differential of $R_p$ at $p \in {\mathcal M}$.
For any $x,p$ on a Hadamard manifold $\HH^n$ we
can write the reflection as
\[
R_p(x) = \exp_p(-\log_px).
\]
The {\it reflection of a proper convex lsc function} $\varphi\colon\HH^n \rightarrow (-\infty,+\infty]$ is given	
by ${\mathcal R}_\varphi\colon \mathcal{H}^n \rightarrow  \mathcal{H}^n$ with
\begin{equation} \label{reflection_hada}
{\mathcal R}_\varphi (x) = \exp_{\prox_{ \varphi }(x)}\bigl(-\log_{\prox_{ \varphi }(x)}(x)\bigr).
\end{equation}

Finally a connected Riemannian manifold ${\mathcal M}$ is called (globally) \emph{symmetric},
if the geodesic reflection at any point $p \in {\mathcal M}$
is an isometry of ${\mathcal M}$, i.e., for all $x,y \in {\mathcal M}$  we have
\begin{equation}
	d\bigl(R_p(x),R_p(y) \bigr) = d(x,y).
\end{equation}
Examples are spheres, connected compact Lie groups, 
Grassmannians, hyperbolic spaces and symmetric positive definite matrices.
The latter two manifolds are symmetric Hadamard manifolds.
For more information we refer to~\cite{E97,H94}.
%
%----------------------------------------------
\section{Parallel DR Algorithm on Hadamard Manifolds} \label{sec:hadamard}
%----------------------------------------------
%
In this section, we generalize the parallel DR algorithm to
Hadamard manifolds. Given two proper convex lsc functions
$\varphi,\psi\colon \HH^n \rightarrow (-\infty,+\infty]$ the DR algorithm aims to
solve
\begin{equation}\label{eq:dr_split_had}
 \argmin_{x \in \mathbb \HH^n} \bigl\{ \varphi(x) + \psi(x) \bigr\}.
\end{equation}
%
%---------------------------------------------------
\begin{algorithm}[t]
	\caption[]{DR Algorithm for Data in Hadamard Manifolds} 
	\label{alg:DR_had}
	\begin{algorithmic}
		\State \textbf{Input:}
		$t^{(0)}\in\HH^n$,
		$\lambda_r \in[0,1]$
		with
		$\sum_{r \in \mathbb N} \lambda_r(1-\lambda_r) = + \infty$, 
		$\eta > 0$
		\State \( r = 0 \);
		\Repeat
		\State $s^{(r)} = {\mathcal R}_{\eta \phi} {\mathcal R}_{\eta \psi} \bigl( t^{(r)} \bigr)$
		\State $t^{(r+1)} = \gamma_{\overset{\frown}{t^{(r)},s^{(r)}}} (\lambda_r)$
		\State $r\rightarrow r+1$;
		\Until a stopping criterion is reached
	\end{algorithmic}
\end{algorithm}
%---------------------------------------------------
%
Adapting Algorithm~\ref{alg:DR_real} to the manifold-valued setting
yields Algorithm~\ref{alg:DR_had}.
If the iterates of Algorithm~\ref{alg:DR_had} converge to a fixed
point~$\hat t \in \HH^n$, then we will see in the next theorem
that
$
\hat x\coloneqq \prox_{\eta\psi} (\hat t)
$
is a solution of~\eqref{eq:dr_split_had}.

%
%-------------------------------------------------------------------------------
\begin{theorem}\label{th:soln}
	Let $\HH$ be a Hadamard manifold and
	$\varphi,\psi\colon\HH^n \rightarrow \RR$ proper, convex, lsc functions
	such that 
	$\interior (\dom \varphi) \cap \dom \psi \not = \emptyset$.
	Assume that a solution of~\eqref{eq:dr_split_had} exists. 
	Then, for each solution $\hat x$ of~\eqref{eq:dr_split_had} there exists
	a fixed point $\hat t$ of $\mathcal{R}_{\eta \varphi}\mathcal{R}_{\eta \psi}$
	such that 
	\[
		\hat x = \prox_{\eta \psi} \hat t.
	\] 
	Conversely, if a fixed point $\hat t$
	of $\mathcal{R}_{\eta \varphi} \mathcal{R}_{\eta \psi}$ exists,
	then $\hat x$ defined as above is a solution of~\eqref{eq:dr_split_had}.
\end{theorem}
%-------------------------------------------------------------------------------
%

The proof is given in Appendix~\ref{app:theo}.

%--------------------------------------------------------------------
Next we are interested in the minimization of functionals containing
multiple summands,
\begin{align} \label{eq:many_had}
\argmin_{x\in \HH^n} \Bigl\{ \sum_{k=1}^K \varphi_k(x) \Bigl\},
\end{align}
where $\varphi_k\colon \HH^n \rightarrow (-\infty,+\infty]$, $k=1,\ldots,K,$ are proper convex lsc functions.
As in the Euclidean case the problem can be rewritten as
\begin{equation}\label{eq:prod_problem_had}
	\argmin_{x\in\HH^{nK}} \bigl\{\Phi(\vect{x})+\iota_{\textsf{D}}(\vect{x}) \bigr\},
\end{equation}
where 
$\Phi (\vect{x})\coloneqq\sum_{k=1}^K \varphi_k(x_k)$, 
$\vect{x}\coloneqq (x_k)_{k=1}^K,$ 
and
\[
{\textsf{D}}\coloneqq\{\vect{x}\in\HH^{nK} \colon x_1=\dots=x_n\in\HH^n\}.
\]
Obviously, ${\textsf{D}}$ is a nonempty, closed convex set so that its
indicator function is proper convex and lsc, see~\cite[p. 37]{B2014}.
Further, we have for any $\eta > 0$ and
$\vect{x} = (x_1,\ldots,x_K)^\tT \in \HH^{nK}$ that
\begin{equation} \label{prox_iota_D_had}
\prox_{\iota_{\textsf{D}}}( \vect{x})
 = \Pi_{\textsf{D}}(\vect{x})
 %= \vecOne_n \otimes \argmin_{x \in \HH^m} \sum_{k=1}^n d(x_k,x)^2.
 = \Big(
% \underbrace{
	\argmin_{x \in \HH^n} \sum_{k=1}^K d(x_k,x)^2, \ldots, \argmin_{x \in \HH^n} \sum_{k=1}^K d(x_k,x)^2
	%}_{K \; \text{times}}
	\Big)
	\in\HH^{nK}
	.
 \end{equation}
The minimizer of the sum is a so-called \emph{Karcher mean}, Fr{\'e}chet mean
or Riemannian center of mass~\cite{Karcher1977}.
It can be efficiently computed on Riemannian manifolds using a gradient descent
algorithm as investigated in~\cite{ATV13} or
on Hadamard manifolds by employing, e.g., a cyclic proximal point algoritm algorithm,
see~\cite{Bac14}.

The parallel DR algorithm for data in a symmetric Hadamard manifold is given by
Algorithm~\ref{alg:DR_par_had}.
%
%---------------------------------------------------
\begin{algorithm}[t]
	\caption[]{Parallel DR Algorithm for Data in Symmetric Hadamard Manifolds} 
	\label{alg:DR_par_had}
	\begin{algorithmic}
		\State \textbf{Input:}  $\vect{t}^{(0)}\in  \HH^{nK}$, $\lambda_r \in[0,1]$ with $\sum_{r \in \mathbb N} \lambda_r(1-\lambda_r) = + \infty$, 
		$\eta > 0$
		\State r = 0;
		\Repeat
		\State $\vect{s}^{(r)} = {\mathcal R}_{\eta \Phi} {\mathcal R}_{\iota_{\textsf{D}}}\bigl( \vect{t}^{(r)} \bigr)$
		\State $\vect{t}^{(r+1)} = \gamma_{\overset{\frown}{\vect{t}^{(r)},\vect{s}^{(r)}}} (\lambda_r)$
		\State $r\rightarrow r+1$;
		\Until a stopping criterion is reached
	\end{algorithmic}
\end{algorithm}
%---------------------------------------------------
%
If the sequence $\{\vect{t}^{(r)}\}_{r \in \mathbb N}$ converges to
some $\vect{\hat t} = (\hat t_1,\ldots,\hat t_K)^\tT$, then we know by Theorem \ref{th:soln} that
\begin{equation}\label{eq:algDRparhad-x}
	\hat x\coloneqq  \argmin_{x \in \HH^n} \sum_{k=1}^K d(\hat t_k,x)^2
\end{equation}
is a solution of~\eqref{eq:many_had}.

%
%-------------------------------------------------------------------------------
\begin{remark} (Cyclic DR Algorithm) \label{rem:cyclic_drs_had}
	Since the cyclic DR iterations can be written similarly to
	Remark~\ref{rem:cyclic_drs} as a Krasnoselski-Mann iteration we have for nonexpansive operators
	$R_{\eta \varphi_k}$, that cyclic DR algorithm converges to some fixed point.
	However, the relation of the fixed point to a solution of the minimization
	problem is, as in the Euclidean case with non indicator functions, completely unknown. 
\end{remark}
%------------------------------------------------------
%

It remains to examine under which conditions the sequence $\{\vect{t}^{(r)}\}_{r \in \mathbb N}$
converges.
Setting $T\coloneqq\mathcal{R}_{\eta \varphi} \mathcal{R}_{\eta \psi}$, the
DR iteration can be seen as a \emph{Krasnoselski–Mann iteration} given
by
\begin{equation}\label{eq:krasnoleski}
 t^{(r+1)}\coloneqq \gamma_{\overset{\frown}{t^{(r)},T(t^{(r)})}} (\lambda_r). 
\end{equation}
The following theorem on the convergence of Krasnoselski–Mann iterations on
Hadamard spaces was proved in~\cite{K2013}, see also~\cite[Theorem 6.2.1]{B2014}.
%
%-------------------------------
\begin{theorem} \label{th:kras}
	Let $(\HH,d)$ be a finite dimensional Hadamard space and
	$T\colon\HH\rightarrow\HH$ a nonexpansive mapping with nonempty fixed point set.
	Assume that $(\lambda_r)_{r\in\NN}$
	satisfies $\sum_{r\in\NN} \lambda_r(1-\lambda_r)=\infty$. 
	Then the sequence $\{t^{(r)}\}_{k\in\NN}$ generated by the 
	Krasnoselski–Mann iterations~\eqref{eq:krasnoleski} converges
	for any starting point  $t^{(0)}\in\HH$ to some fixed point $\hat t$ of $T$.
%	If $\HH$ is not finite dimensional only weak convergence can be ensured.
\end{theorem}
%-------------------------------
%
Hence, if the operator $T = \mathcal{R}_{\eta \Phi}\mathcal{R}_{\iota_{\textsf{D}}}$ is
nonexpansive the convergence of the parallel DR algorithm to some fixed point is
ensured. Obviously, this is true if  $\mathcal{R}_{\eta \varphi_k}$, $k=1,\ldots,K$ 
and $\mathcal{R}_{\iota_{\textsf{D}}}$ are nonexpansive.
Unfortunately, the result from the Euclidean space that proper convex lsc functions 
produce nonexpansive reflections does not carry over the Hadamard manifold setting.
However, we show in the following sections that the parallel DR algorithm is well suited for finding the
minimizer of the ROF-like functional on images with values in a symmetric Hadamard manifold.

%
%----------------------------------------------
\section{Application of the DR Algorithm to ROF-like Functionals}  \label{sec:rof}
%----------------------------------------------
%
One of the most frequently used variational models for restoring noisy images or images with
missing pixels is the ROF model~\cite{ROF92}. 
In a discrete,  anisotropic setting this model is given by~\eqref{rof_real}.
A generalization of the ROF model to images having cyclic values was proposed in~\cite{SC11} 
and to images with pixel values in a manifolds in~\cite{LSKC13,WDS2014}.
The model looks as follows:
let
${\mathcal G} \coloneqq \{1,\ldots,N\} \times \{1,\ldots,M\}$
be the image grid and
$\emptyset \not = {\mathcal V} \subseteq {\mathcal G}$.
Here \(\mathcal V\) denotes the set of available, in general noisy
pixels. In particular, we have ${\mathcal V} = {\mathcal G}$ in the case 
of no missing pixels.
Given an image
$f\colon {\mathcal V} \rightarrow {\HH}$
which is corrupted by noise or missing pixels,
we want to restore the original image 
$u_0\colon {\mathcal G} \rightarrow {\HH}$
as a minimizer of the functional
\begin{equation}\label{task}
	\begin{split}
{\mathcal E}(u) &= \mathcal D(u; f) + \alpha\TV(u)\\
 &\coloneqq\frac{1}{2}
		 \sum_{(i,j) \in {\mathcal V}} d (f_{i,j},u_{i,j})^2
	+ \alpha %\underbrace{
	\biggl(
		\sum_{(i,j) \in {\mathcal G}} d (u_{i,j},u_{i+1,j}) 
		+\sum_{(i,j) \in {\mathcal G}} d (u_{i,j},u_{i,j+1})
	\biggr)
	%}_{\operatorname{TV}(u)},		
	\end{split}
\end{equation}
where $\alpha >0$ and we assume mirror boundary conditions, i.e.,
$d(u_{i,j},u_{i+1,j}) = 0$ if $(i+1,j) \not \in {\mathcal G}$ and
$d(u_{i,j},u_{i,j+1}) = 0$ if $(i,j+1) \not \in {\mathcal G}$.
By property~\ref{item:DandDSq-convex} the functional ${\mathcal E}$ is convex.
If ${\mathcal V}\not = \emptyset$ then, regarding that the TV regularizer has only
constant images in its kernel, 
we see by Theorem~\ref{hadamardmin}\,\ref{item:pclsc-minimizer} that
${\mathcal E}$ has a global minimizer.
If ${\mathcal V} = {\mathcal G}$, then, by~\ref{item:DSq-stronly-convex} the functional
is strongly convex and has  a unique minimizer by
Theorem~\ref{hadamardmin}\,\ref{item:psclsc-unique}.

Reordering the image columnwise into a vector of length $n \coloneqq NM$ we can consider
$f$ and $u$ as elements in the product space $\HH^n$ which fits into the notation of the previous section.
However, it is more convenient to keep the 2D formulation here.
Denoting the data fidelity term by 
\[
\varphi_1 (u) = D(u;f) = \sum_{(i,j) \in {\mathcal V}} \tfrac12 d(f_{i,j},u_{i,j})^2
\]
and splitting the regularizing term \(\TV(u)\) by grouping the odd and even indices with respect to both image dimensions
as
\begin{align} \label{eq:splitTVevodd}
	 \alpha \TV(u)
	&=	
\alpha \sum_{\nu_1=0}^1 \sum_{i,j=1}^{\bigl\lfloor\!\frac{N-\nu_1}{2}\!\bigr\rfloor,M}
		 d (u_{2i-1+\nu_1,j},u_{2i+\nu_1,j})
\\&\qquad+ \alpha \sum_{\nu_2=0}^1
	 \sum_{i,j=1}^{N,\bigl\lfloor\!\frac{M-\nu_2}{2}\!\bigr\rfloor}
		 d (u_{i,2j-1+\nu_2},u_{i,2j+\nu_2}) \\
	&= \sum_{k=2}^5 \varphi_k(u)
\end{align}
our functional becomes
\begin{align}\label{task_sum_TV}
{\mathcal E}(u) = \sum_{k=1}^5 \varphi_k(u).
\end{align}
This has exactly the form \eqref{eq:many_had} with $K=5$. We want to apply the parallel DR algorithm.
To this end we have to compute the proximal mappings of the $\varphi_k$, $k=1,\ldots,K$.
The reason for the above special splitting is that every pixel appears in each
functional \(\varphi_k\) at most in one summand so that the proximal values can be computed componentwise.
More precisely, we have for $\varphi_1$ and $x \in \HH^{N,M}$ that
\begin{align}
\prox_{\eta \varphi_1} (x) 
&= 
\argmin_{u \in \HH^{N,M} } \Big\{ \frac12 \sum_{(i,j) \in {\mathcal G}} d( x_{i,j},u_{i,j} )^2 
+  \frac{\eta}{2} \sum_{(i,j) \in {\mathcal V}} d( f_{i,j},u_{i,j} )^2 \Big\}\\
&=
\Big( \prox_{\eta g_{i,j}} (x_{i,j}) \Big)_{(i,j) \in {\mathcal G}},
\end{align}
where $g_{i,j}\colon \HH \rightarrow \mathbb R$ is defined by 
\begin{equation}\label{klein_g}
g_{i,j} \coloneqq \frac12 d(f_{i,j},\cdot)^2
\end{equation} 
if $(i,j) \in {\mathcal V}$ and $g_{i,j} \coloneqq 0$ otherwise.
\\
For $\varphi_2$ we get
\begin{align}
 \prox_{\eta {\varphi}_2 }(x) 
 &= \argmin_{u \in {\HH^{N,M}}} \sum_{i,j=1}^{N,M} \tfrac12 d (x_{i,j},u_{i,j}) ^2 + \eta
 \sum_{i,j=1}^{\bigl\lfloor\!\frac{N}{2}\!\bigr\rfloor,M}
		 d (u_{2i-1,j},u_{2i,j})\\
 &=\argmin_{u \in {\HH^{N,M}}} \sum_{i,j=1}^{\bigl\lfloor\!\frac{N}{2}\!\bigr\rfloor,M} \left( \tfrac12 d (x_{2i-1,j},u_{2i-1,j}) ^2 + \tfrac12 d (x_{2i,j},x_{2i,j}) ^2
		 + \eta d (u_{2i-1,j},u_{2i,j}) \right)\\
 &= \left( \prox_{\eta G }(x_{2i-1,j},x_{2i,j}) \right)_{i,j=1}^{\bigl\lfloor\!\frac{N}{2}\!\bigr\rfloor,M}, 
\end{align}
where $G\colon \HH^2 \rightarrow \mathbb R$ is given by 
\begin{equation}\label{gross_G}
G \coloneqq d(\cdot,\cdot).
\end{equation}
Similarly the proximal mappings of $\varphi_k$, $k=3,4,5$, can be computed for pairwise components.

The proximal mappings of $g$ and $G$, and consequently of our functions $\varphi_k$, $k=1,\ldots,K$, 
are given analytically, see~\cite{FO2002,WDS2014}.

%-------------------------------------
\begin{lemma} \label{lem:proxies}
Let $(\HH,d)$ be an Hadamard manifold, $\eta > 0$ and $a \in \HH$.
\begin{enumerate}[label={\textup{\roman*)}}]
	\item\label{item:dataProx}
	The proximal mappings of $g\coloneqq \frac{1}{\nu} d(a,\cdot)^\nu$, $\nu \in \{1,2\}$,
		are given by 
		\begin{equation}
			p_x = \prox_{\eta g}(x) = \gamma_{\overset{\frown}{x,a}}(\hat s)
		\end{equation}
		with
		\begin{align}
			\hat s \coloneqq
			\begin{cases}
				\min\big\{ \tfrac{\eta}{d(x,a)} ,1 \big\}&\mbox{\textup{ for }} \nu = 1,\\
				\tfrac{\eta}{1+\eta} & \mbox{\textup{ for }} \nu = 2.
			\end{cases}
\end{align}
	\item\label{item:diffProx}
		The proximal mappings of
		$G\coloneqq d(\cdot,\cdot)^\nu$, $\nu \in \{1,2\}$,
		are given for $x = (x_0,x_2)$ by 
		\begin{equation}
			p_x = \prox_{\eta G} (x)
			= \bigl(\gamma_{\overset{\frown}{x_0,x_1}}(\hat s),\gamma_{\overset{\frown}{x_1,x_0}}(\hat s)\bigr),
		\end{equation}
		with
		\begin{align}
			\hat s \coloneqq
			\begin{cases}
				\min\big\{ \tfrac{\eta}{d(x_0,x_1)} , \tfrac{1}{2} \big\}&\mbox{\textup{ for }} \; \nu = 1,\\
				\tfrac{\eta}{1+2\eta}&\mbox{\textup{ for }}\nu = 2.
			\end{cases}
		\end{align}	
\end{enumerate}
\end{lemma}
%-------------------------------------
%

It follows that the reflection of $\varphi_1$ is just given by the componentwise reflections
of the $g_{i,j}$, i.e., for $x \in \HH^{N,M}$ we have 
\[{\mathcal R}_{\varphi_1}(x) =  \left( {\mathcal R}_{g_{i,j}}(x_{i,j}) \right)_{i,j \in {\mathcal G}}.\]
Similarly, the reflection of $\varphi_k$, $k=2,\ldots,5$, are determined by pairwise reflections
of $G$, e.g., 
\[{\mathcal R}_{\varphi_2}(x) =  \left( {\mathcal R}_{G}(x_{2i-1,j},x_{2i,j}) \right)_{i,j=1}^{\bigl\lfloor\!\frac{N}{2}\!\bigr\rfloor,M}.\]

Based on these considerations and Theorem~\ref{th:kras} we see immediately that
the parallel DR algorithm~\ref{alg:DR_par_had} applied to the special splitting~\eqref{task_sum_TV}
of the ROF-like functional on images with values in $\HH$ converges if the reflections
${\mathcal R}_{g_{i,j}}$, 
${\mathcal R}_{G}$ and ${\mathcal R}_{\iota_{\textsf{D}}}$
are nonexpansive.
This is covered in the next two sections.
%
%----------------------------------------------------------------
\section{Reflections of Functions Related to Distances} \label{subsec:dist}
%----------------------------------------------------------------
%
In this section we will show that reflections with respect to the distance-like functions
$g$ and $G$ in~\eqref{klein_g} and~\eqref{gross_G} are indeed nonexpansive
which is an interesting result on its own.

We will  need the following auxiliary lemma.

%
%---------------------------------------------------------
\begin{lemma} \label{est_bacak}
Let $(\HH,d)$ be an Hadamard space. Then for $x_0,x_1,y_0,y_1 \in \HH$
it holds
\begin{equation}\label{formel_w3}
	\begin{split}
	0 &\le d^2(x_0,y_0) + d^2(x_1,y_1) + d^2(x_0,x_1)+d^2(y_0,y_1)
	\\
	&\qquad -d^2(x_0,y_1) - d^2(x_1,y_0).		
	\end{split}
\end{equation}
For $s,t \in [0,1]$ we have
\begin{align} \label{formel_w1}
	d^2\bigl(\gamma_{\overset{\frown}{x_0,x_1}}(s),
		\gamma_{\overset{\frown}{y_0,y_1}}(t)\bigr)
	&\le st d^2(x_1,y_1)+(1-s)(1-t)d^2(x_0,y_0)\\
	&\quad + (1-s)td^2(x_0,y_1)+s(1-t)d^2(x_1,y_0)\\
	&\quad - s(1-s)d^2(x_0,x_1)-t(1-t)d^2(y_0,y_1).
\end{align}
For $s=t$, this becomes 
\begin{align} \label{formel_w2}
	d^2\bigl(\gamma_{\overset{\frown}{x_0,x_1}}(s),
		\gamma_{\overset{\frown}{y_0,y_1}}(s)\bigr)
	&\le s^2 d^2(x_1,y_1)+(1-s)^2d^2(x_0,y_0)\\
	& + (1-s)s\bigl(d^2(x_0,y_1)+ d^2(x_1,y_0) - d^2(x_0,x_1)-d^2(y_0,y_1) \bigr).
\end{align}
and for $x=x_0$, $y = y_0$ and $x_1=y_1 = a$ we obtain further
\begin{align} \label{formel_w4}
d^2 \bigl(\gamma_{\overset{\frown}{x,a}}( s),\gamma_{\overset{\frown}{y,a}}(t) \bigr) 
&\le
d^2(x,y) + s(s-2) d^2(x,a) + t(t-2) d^2(y,a) \\
&\quad + 2 (s+t-st) d(x,a)d(y,a).
\end{align}
\end{lemma}

%-------------------------------------------------------------------------------
\begin{proof}
Estimate~\eqref{formel_w3} was proved in~\cite[Corollary 1.2.5]{B2014}.
Relation~\eqref{formel_w1} can be deduced by applying (D2) twice.
Formula~\eqref{formel_w2} follows directly from the previous one by setting $s=t$.
It remains to show \eqref{formel_w4}.
For $x=x_0$, $y = y_0$ and $x_1=y_1 = a$, inequality~\eqref{formel_w1} becomes 
\begin{align}
d^2 \big(&\gamma_{\overset{\frown}{x,a}}( s),\gamma_{\overset{\frown}{y,a}}(t) \big)
\\&\le
(1-s) (1-t) d^2(x,y) + (1-s) (t-s) d^2(x,a) - (1-t)(t-s) d^2(y,a) 
\\
&= d^2(x,y) - (s+t-st)  d^2(x,y)
 + (1-s) (t-s) d^2(x,a) - (1-t)(t-s) d^2(y,a).
\end{align}
By the triangle inequality we have
\begin{align}
d(x,y) &\ge |d(x,a) - d(y,a)|,\\
d^2(x,y) &\ge d^2(x,a) + d^2(y,a) - 2 d(x,a)d(y,a).\label{triang}
\end{align}
Since $s,t \in [0,1]$, it holds $s+t-st \ge 0$. Replacing $(s+t-st) d^2(x,y)$ 
we obtain
\begin{align}
d^2 \big(\gamma_{\overset{\frown}{x,a}}( s),\gamma_{\overset{\frown}{y,a}}(t) \big) 
&\le
d^2(x,y) - (s+t-st) \left( d^2(x,a) + d^2(y,a) - 2 d(x,a)d(y,a) \right)\\
&\quad + (1-s) (t-s) d^2(x,a) - (1-t)(t-s) d^2(y,a).
\end{align}
Resorting yields~\eqref{formel_w4}.
\end{proof}
%------------------------------------------------------------------------

%
% Fig 1
\begin{figure}[t]\centering
	\begin{subfigure}{0.49\textwidth}\centering
		\includegraphics{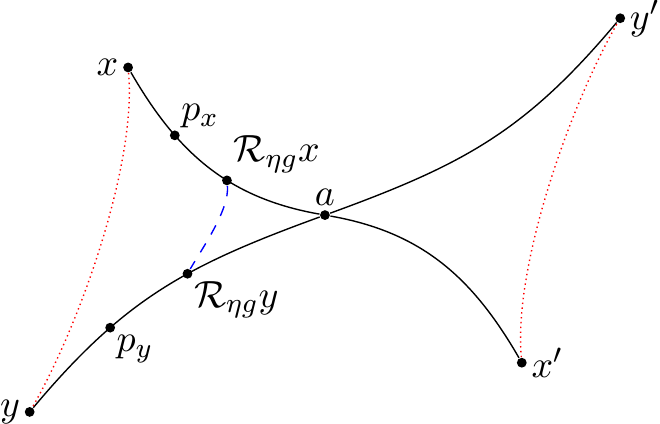}			
		\caption[]{$\hat{s}= \tfrac{\eta}{1+\eta}\in\bigl[0,\tfrac{1}{2}\bigr]$}\label{fig:d2data:case1}
	\end{subfigure}	
	\begin{subfigure}{0.49\textwidth}\centering
		\includegraphics{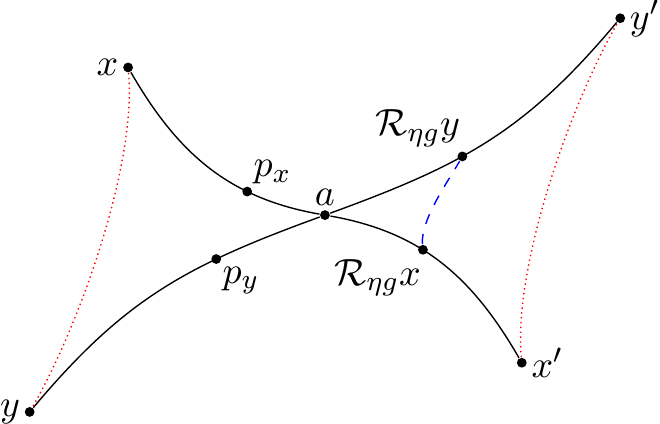}				
		\caption[]{$\hat{s}= \tfrac{\eta}{1+\eta}\in\bigl(\tfrac{1}{2},1\bigr]$}\label{fig:d2data:case2}
	\end{subfigure}
	\caption[]{Illustration of the nonexpansiveness of the reflections in Theorem \ref{reflex:d(xa)} for $g(x) = d^2(x,a)$.}
\end{figure}
\begin{figure}[t]\centering
	\begin{subfigure}{0.49\textwidth}\centering
		\includegraphics{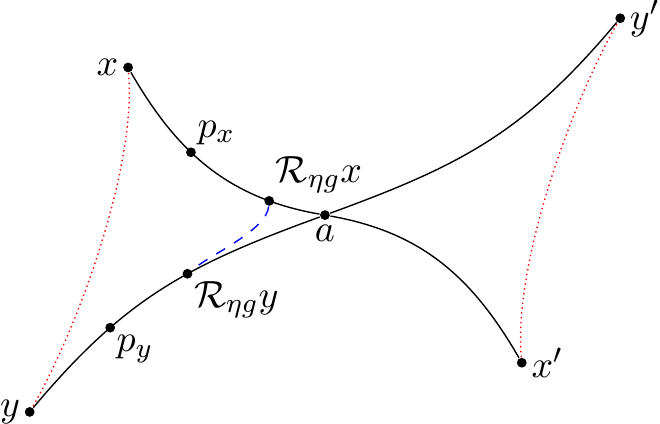}	
		\caption[]{$\hat{s},\hat{t}\in\bigl[0,\tfrac{1}{2}\bigr]$}\label{fig:d1data:case1}
	\end{subfigure}	
	\begin{subfigure}{0.49\textwidth}\centering
		\includegraphics{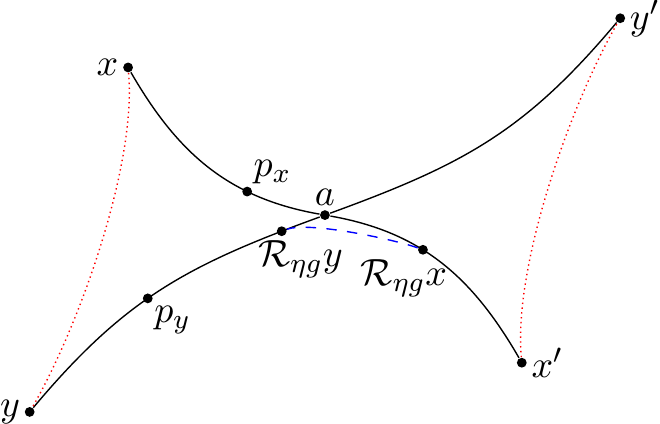}
		\caption[]{$\hat{s}\in\bigl[0,\tfrac{1}{2}\bigr],\hat{t}\in\bigl(\tfrac{1}{2},1\bigr]$}\label{fig:d1data:case3}
	\end{subfigure}
	\caption[]{Illustration of the nonexpansiveness of the reflections in Theorem \ref{reflex:d(xa)} for $g(x) = d(x,a)$, where $\hat{s} =\min\big\{ \tfrac{\eta}{d(x,a)} ,1 \big\}$ and $\hat{t} = \min\big\{ \tfrac{\eta}{d(y,a)} ,1 \big\}$.}
\end{figure}

Now we consider the reflections. 
We will use that for two points \(x,a\in\mathcal H\) and
$p =  \gamma_{\overset{\frown}{x,a}}( \hat t)$, $\hat t \in [0,1]$,
the reflection of $x$ at $p$ can be written as
\begin{equation} \label{eq:reflect_2}
R_p(x) = 
\begin{cases}
	\gamma_{\overset{\frown}{x,a}}( 2\hat t\,)&\mbox{ if } 2\hat t \in [0,1],\\
	\gamma_{\overset{\frown}{x', a}}( 2- 2\hat t\,)&\mbox{ if } 2\hat t \in (1,2]
\end{cases}
\end{equation}
with $x'\coloneqq R_a(x)$. 

%
%-------------------------------------------------------
\begin{theorem} \label{reflex:d(xa)}
For an arbitrary fixed $a \in {\mathcal H}$ and $g \coloneqq \frac{1}{\nu} d^\nu(a,\cdot)$,
$\nu \in\{1,2\}$,
the reflection ${\mathcal R}_{\eta g}$, $\eta >0$, is nonexpansive.
\end{theorem}
%-------------------------------------------------------
%

\begin{proof}
By~\eqref{eq:reflect_2} we have 
\[
	{\mathcal R}_{\eta g}(x) = R_{p_x} (x) =
	\begin{cases}
		\gamma_{\overset{\frown}{x,a}}( 2\hat s)&\mbox{ if } 2\hat s \in [0,1],\\
		\gamma_{\overset{\frown}{x', a}}( 2- 2\hat s)&\mbox{ if } 2\hat s \in (1,2]
	\end{cases}
\] 
where $x' \coloneqq R_a(x)$, and $p_x$ resp.~$\hat s$ are given by Lemma~\ref{lem:proxies}\,\ref{item:dataProx}.
Similarly, we get
$p_y = \prox_{\eta g}(y) = \gamma_{\overset{\frown}{y,a}}(\hat t)$, where 
$\hat t = \hat s$ if $\nu = 2$ and
$\hat t = \min\big\{ \tfrac{\eta}{d(y,p)} ,1 \big\}$ if $\nu = 1$.
\begin{enumerate}[label={\arabic*.},leftmargin=0pt,itemindent=*]
	%
	% 1.
	\item Let $\nu = 2$. We distinguish two cases. 
	\begin{enumerate}[label={\arabic{enumi}.\arabic*.},leftmargin=0pt,itemindent=*]
	\item
	If
	$2\hat s \in [0,1]$, cf. Fig.~\ref{fig:d2data:case1}, we obtain with $s\coloneqq 2 \hat s$ by the joint convexity of the distance
	function
	\begin{align}
		d\bigl(\mathcal{R}_{\eta g}(x),\mathcal{R}_{\eta g}(y)\bigr)
		= d \bigl(\gamma_{\overset{\frown}{x,a}}( s),
			\gamma_{\overset{\frown}{y,a}}( s) \bigr)
		\le (1- s)d(x,y) \le d(x,y).
	\end{align}
	\item
	In the case $2\hat s \in (1,2]$, see Fig.~\ref{fig:d2data:case2}, we verify with
	$s\coloneqq 2- 2\hat s \in (0,1]$
	by the joint convexity of the distance function
	and since the reflection is an isometry that
	\begin{align}
		d\bigl(\mathcal{R}_{\eta g}(x),\mathcal{R}_{\eta g}(y)\bigr)
		= d \bigl(\gamma_{\overset{\frown}{x',a}}( s),
			\gamma_{\overset{\frown}{y',a}}( s) \bigr)
		\le (1- s)d\bigl(R_a(x),R_a(y)\bigr) \le d(x,y).
	\end{align}
	\end{enumerate}
	%
	% 2.
	\item Let $\nu = 1$. Without loss of generality we assume that
	$d(x,a) \le d(y,a)$ so that $\tfrac{\eta}{d(y,a)} \le \tfrac{\eta}{d(x,a)}$.
	We distinguish three cases.
	\begin{enumerate}[label={\arabic{enumi}.\arabic*.},leftmargin=0pt,itemindent=*]
		%
		% 2.1
		\item In the case $d(x,a) \le d(y,a) \le \eta$ we obtain
		$\hat s = \hat t = 1$ and consequently we set
		$s \coloneqq 2 - 2 \hat s =0$ and $t \coloneqq 2 - 2 \hat t =0$.
		Since the reflection is an isometry this implies
		\begin{align}
			d\bigl(\mathcal{R}_{\eta g}(x),\mathcal{R}_{\eta g}(y)\bigr)
			=
			d \bigl(\gamma_{\overset{\frown}{x',a}}(0),
			\gamma_{\overset{\frown}{y',a}}(0) \bigr)
			=
			d\bigl(R_a(x),R_a(y)\bigr)
			=
			d(x,y).
		\end{align}
		%
		% 2.2
		\item In the case $\eta \le d(x,a) \le d(y,a)$ we have
		$\hat s = \tfrac{\eta}{d(x,a)}$ and $\hat t = \tfrac{\eta}{d(y,a)}$.
		We distinguish three cases.
		\begin{enumerate}[label={\roman*)},leftmargin=1em,itemindent=*]
			%
			% 2.2 i)
			\item\label{item:bothDist-latnu-hatshatt-lethalf}
			If $\hat t, \hat s \le \tfrac12$, cf.~Fig.~\ref{fig:d1data:case1}, we obtain with
			$s \coloneqq 2\hat s$ and $t \coloneqq 2\hat t$
			by~\eqref{formel_w4} the estimate
			\begin{align} \label{eq_basic}
				d^2\bigl(\mathcal{R}_{\eta g}(x),\mathcal{R}_{\eta g}(y)\bigr)
				&= 
				d^2\bigl(\gamma_{\overset{\frown}{x,a}}( s),
				\gamma_{\overset{\frown}{y,a}}(t) \bigr) \\
				&\le d^2(x,y) + s(s-2) d^2(x,a) + t(t-2) d^2(y,a)\\
				&\quad + 2(s+t-st) d(x,a) d(y,a).
			\end{align}
			Plugging in $s = \tfrac{2\eta}{d(x,a)}$
			and $t = \tfrac{2\eta}{d(y,a)}$ we see that the expressions
			containing $s$ and $t$ vanish so that the reflection is non-expansive.
			%
			% 2.2. ii)
			\item \label{}
			If $\hat t \le  \tfrac12 \le \hat s$, cf.~Fig.~\ref{fig:d1data:case3}, we get
			with $s \coloneqq 2-2 \hat s$ and $t \coloneqq 2\hat t$ that
			\begin{align}
				d\bigl(\mathcal{R}_{\eta g}(x),\mathcal{R}_{\eta g}(y)\bigr)
				&=
				d \big(\gamma_{\overset{\frown}{x',a}}( s),
					\gamma_{\overset{\frown}{y,a}}(t) \big)
				=
				d(\tilde x,\tilde y)
			\end{align}
			with
			$\tilde x = \gamma_{\overset{\frown}{x',a}}( s)$
			and $\tilde y = \gamma_{\overset{\frown}{y,a}}(t)$.
			By the assumption we have
			$d(a,\tilde x) = 2\eta - d(x,a)$
			and
			$d(y,a) = 2\eta + d(a,\tilde y)$. Further, it follows by the
			triangle inequality
			\begin{align}
				d(\tilde y,\tilde x)
				&\le  d(a,\tilde x) + d(a,\tilde y) = 2\eta - d(x,a) + d(a,\tilde y),\\
				d(x,y)
				&\ge d(y,a) - d(x,a) = 2\eta + d(a,\tilde y) - d(x,a).
			\end{align}
			Hence we conclude $d(\tilde y,\tilde x) \le d(x,y)$.
			%
			% 2.2. iii)
			\item
			If $\hat t, \hat s \ge \tfrac12$ we
			set $s \coloneqq 2-2 \hat s$ and $t \coloneqq 2 - 2\hat t$
			and obtain by~\eqref{formel_w4} and the isometry property of the
			reflection
			\begin{align} \label{22iii}
				d^2\bigl(\mathcal{R}_{\eta g}(x),\mathcal{R}_{\eta g}(y)\bigr)
				&=
				d^2 \bigl(\gamma_{\overset{\frown}{x',a}}( s),
					\gamma_{\overset{\frown}{y',a}}(t) \bigr)
				\\
				&\le
				d^2(x,y) + s(s-2) d^2(x,a) + t(t-2) d^2(y,a)\\
				&\quad + 2 (s+t-st) d(x,a)d(y,a).
			\end{align}
			Plugging in the expressions for $s$ and $t$ and noting that the above expression is 
			symmetric with respect to $(s,t)$ and $(2-s,2-t)$ we obtain the assertion as in~\ref{item:bothDist-latnu-hatshatt-lethalf}.
		\end{enumerate} %end of subcases of 2.2
		% 2.3)
		\item 
		In the case $d(x,a) \le \eta \le d(y,a)$ we have $\hat s = 1$
		and $\hat t = \tfrac{\eta}{d(y,a)}$. Again, we distinguish two cases:
		\begin{enumerate}[label={\roman*)},leftmargin=1em,itemindent=*]
			%
			% 2.3. i)
			\item If $\hat t \le \tfrac12$ we use $t\coloneqq 2 \hat t$ and  obtain
			\begin{align}
				d\bigl(\mathcal{R}_{\eta g}(x),\mathcal{R}_{\eta g}(y)\bigr) 
				&=
				d \bigl(\gamma_{\overset{\frown}{x',a}}(0),
					\gamma_{\overset{\frown}{y,a}}(t) \bigr)
				=
				d \big(x',\tilde y \big)
			\end{align}
			with $\tilde y\coloneqq \gamma_{\overset{\frown}{y,a}}(t)$.
			By assumption we have
			$2d(x,a) \le 2\eta$ so that
			\[
				d(a,y)  \ge 2 d(a,x) + d(a,\tilde y ).
			\]
			By the triangle inequality it follows
			\begin{align}
				d (x',\tilde y ) 
				&\le d(x',a) + d (a,\tilde y )
				= d(x,a) + d (a,\tilde y),\\
				d(x,y) &\ge d(a,y) - d(a,x) \ge d(a,x) + d(a,\tilde y ) 
			\end{align}
			and consequently $ d (x',\tilde y ) \le d(x,y)$.
			%2.3 ii)\\
			\item If $\hat t > \tfrac12$ we put $t = 2- 2 \hat t$ and get
			by~\eqref{formel_w4} and since the reflection is an isometry
			\begin{align}
				d^2\bigl(\mathcal{R}_{\eta g}(x),\mathcal{R}_{\eta g}(y)\bigr)
				&=
				d^2 \bigl(\gamma_{\overset{\frown}{x',a}}(0),
					\gamma_{\overset{\frown}{y',a}}(t) \bigr)\\
				&\le
				d^2(x,y) + 2 t d(x,a)d(y,a) +  t(t-2) d^2(y,a)\\
				&= d^2(x,y) + 2td(y,a)(d(x,a) - \eta) \le d^2(x,y).
			\end{align}
		\end{enumerate} % subcases of 2.3
	\end{enumerate} %subcases of 2
\end{enumerate} %end of both nu cases
This finishes the proof. 
\end{proof}
%-----------------------------------------------

%
%
\begin{figure}[t]
	\centering
	\begin{subfigure}{0.49\textwidth}
		\includegraphics{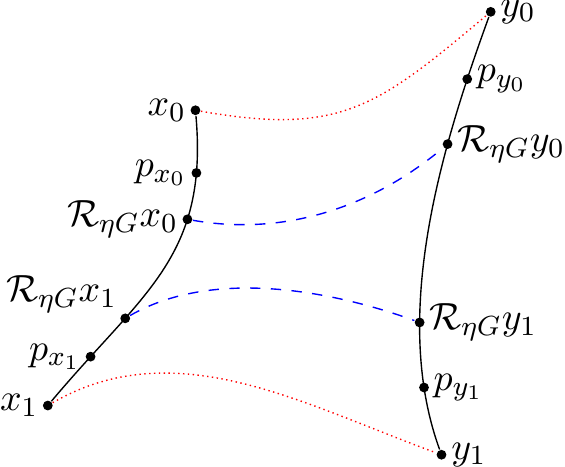}
		\caption[]{$d^2(x_1,x_2)$, $\hat{s}\in\bigl[0,\tfrac{1}{4}\bigr]$}\label{fig:d2tv:case1}
	\end{subfigure}
	\begin{subfigure}{0.49\textwidth}\centering
		\includegraphics{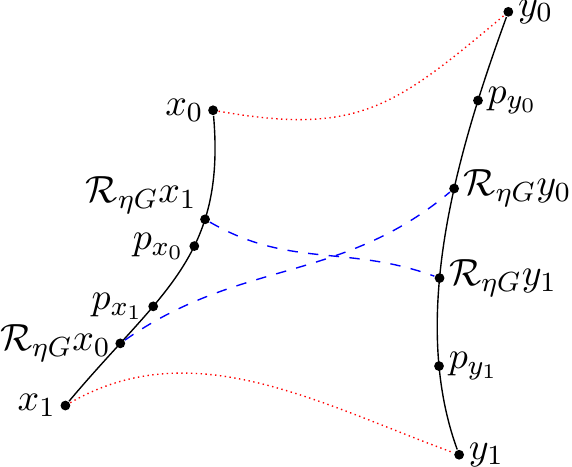}	
		\caption[]{$d(x_1,x_2),\ \hat{s}\in\bigl(\tfrac{1}{4},\ \tfrac{1}{2}\bigr],\ \hat{t}\in\bigl[0,\tfrac{1}{4}\bigr]$ }\label{fig:d2tv:case2}
	\end{subfigure}
	\caption[]{Illustration of the nonexpansiveness of the reflections in Theorem \ref{prop:had_2}, where $p_x = \bigr(p_{x_0},p_{x_1}\bigr)$ and $p_y = \bigr(p_{y_0},p_{y_1}\bigr)$, see Lemma \ref{lem:proxies}. \subref{fig:d2tv:case1} case \ref{prop:had_2:case1}, i.e., $\nu = 2$ and $\hat{s}\in\bigl[0,\tfrac{1}{4}\bigr]$, \subref{fig:d2tv:case2} case \ref{prop:had_2:case23}, i.e., $\nu = 1$, $\hat{s}=\tfrac{1}{2}$, and $\hat{t}\in\bigl[0,\tfrac{1}{4}\bigr]$.}\label{fig:d2tv}
\end{figure}

Recall that the distance $d_{\HH^2}$ on $\HH^2={\mathcal H} \times {\mathcal H}$ is defined by
\begin{align}
	d_{\HH^2}\bigl((x_0,x_1),(y_0,y_1)\bigr)\coloneqq \sqrt{d^2(x_0,y_0)+d^2(x_1,y_1)}.
\end{align}
%
%-------------------------------------------------------
\begin{theorem} \label{prop:had_2}
	For $G\coloneqq d^\nu (\cdot,\cdot)$, $\nu \in \{1,2\}$, the reflection $\mathcal{R}_{\eta G}$, $\eta>0$, is nonexpansive.
\end{theorem}
%-------------------------------------------------------
%

\begin{proof}
By~\eqref{eq:reflect_2} and since $s\coloneqq 2 \hat s \in [0,1]$ we have 
\begin{equation} \label{reflect_10}
	{\mathcal R}_{\eta G}(x)
	= R_{p_x} (x)
	= \big(\gamma_{\overset{\frown}{x_0,x_1}}(s),
		\gamma_{\overset{\frown}{x_1,x_0}}(s) \big),
\end{equation}
where $p_x$ is given by Lemma~\ref{lem:proxies}\,\ref{item:diffProx}.
Similarly, we conclude for $y \coloneqq (y_0,y_1)$
with $\hat t = \hat s$ if $\nu = 2$ and
$\hat t = \min\big\{ \tfrac{\eta}{d(y_0,y_1)} , \tfrac{1}{2} \big\}$
if $\nu = 1$. 
\begin{enumerate}[label={\arabic*.},leftmargin=0pt,itemindent=*]
% 1.
\item \label{prop:had_2:case1} Let $\nu = 2$, cf.~Fig.~\ref{fig:d2tv:case1}. We obtain
	\begin{align}
		d_{\HH^2}^2 \bigl(\mathcal{R}_{\eta G}(x),\mathcal{R}_{\eta G}(y) \bigr) 
		=
		d^2 \big(\gamma_{\overset{\frown}{x_0,x_1}}(s),
			\gamma_{\overset{\frown}{y_0,y_1}}(s) \big)
			+
		d^2 \big(\gamma_{\overset{\frown}{x_1,x_0}}(s),
			\gamma_{\overset{\frown}{y_1,y_0}}(s) \big)
	\end{align}
	and applying~\eqref{formel_w2} we further have
	\begin{align}
		d^2_{\HH^2}&\bigl(\mathcal{R}_{\eta G}(x),\mathcal{R}_{\eta G}(y)\bigr) 
		\\[0.25\baselineskip]
		&\le\bigl(s^2+(1-s)^2\bigr)\bigl(d^2(x_0,y_0)+d^2(x_1,y_1)\bigr)\\
		&\quad+2(1-s)s\bigl(d^2(x_0,y_1)+d^2(x_1,y_0)-d^2(x_0,x_1)-d^2(y_0,y_1)\bigr)
		\\[0.25\baselineskip]
		&=d^2(x_0,y_0)+d^2(x_1,y_1) + 2s(s-1) \bigl(d^2(x_0,y_0)+d^2(x_1,y_1)\bigr)\\
		&\quad-2s(s-1)\bigl(d^2(x_0,y_1)+d^2(x_1,y_0)-d^2(x_0,x_1)-d^2(y_0,y_1)\bigr)
		\\[0.25\baselineskip]
		&= d^2_{\HH^2} (x,y) - 2s(1-s)\\
		&\quad\times\bigl(d^2(x_0,y_0)+d^2(x_1,y_1) + d^2(x_0,x_1)+d^2(y_0,y_1) - d^2(x_0,y_1)-d^2(x_1,y_0)\bigr).
	\end{align}
	By~\eqref{formel_w3} we know that the last factor is non-negative.
	Since $s \in [0,1]$ we conclude $2s(1-s) \ge 0$ so that
	\[
		d^2_{\HH^2}\bigl(\mathcal{R}_{\eta G}(x),\mathcal{R}_{\eta G}(y)\bigr)
		\le d^2_{\HH^2} (x,y).
	\]
	% 2.
	\item Let $\nu = 1$. Without loss of generality we assume
	that $d(x_0,x_1) \le d(y_0,y_1)$ and
	consequently $\tfrac{\eta}{d(y_0,y_1)} \le \tfrac{\eta}{d(x_0,x_1)}$.
	We distinguish three cases.
	\begin{enumerate}[label={\arabic{enumi}.\arabic*.},leftmargin=0pt,itemindent=*]
	% 2.1
	\item  
	If $d(x_0,x_1) \le d(y_0,y_1) \le 2 \eta$ we
	have $s = 1$ so that by~\eqref{reflect_10} it holds
	\begin{equation}
		d_{\HH^2}^2 \bigl(\mathcal{R}_{\eta G} (x),\mathcal{R}_{\eta G} (y) \bigr) 
		=
		d^2(x_1,y_1) + d^2(x_0,y_0) =d_{\HH^2}^2(x,y).
	\end{equation}
	%2.2
	\item \label{prop:had_2:case23} If $2 \eta \le d(x_0,x_1) \le d(y_0,y_1)$, cf.~Fig.~\ref{fig:d2tv:case2}, we use~\eqref{formel_w1}
	to obtain 
	\begin{align}
		d^2_{\HH^2}&\bigl(\mathcal{R}_{\eta G}(x),\mathcal{R}_{\eta G}(y)\bigr) \\
		&= d^2\bigl(\gamma_{\overset{\frown}{x_0,x_1}}(s),\gamma_{\overset{\frown}{y_0,y_1}}(t)\bigr)+
		d^2\bigl(\gamma_{\overset{\frown}{x_1,x_0}}(s),\gamma_{\overset{\frown}{y_1,y_2}}(1-t)\bigr)\\[0.2\baselineskip]
		&\le \bigl(st+(1-s)(1-t)\bigr)\bigl(d^2(x_0,y_0)+d^2(x_1,y_1)\bigr)\\
		&\quad +\bigl((1-s)t+s(1-t)\bigr)\bigl(d^2(x_0,y_1)+d^2(x_1,y_0)\bigr)\\
		&\quad -2(s-s^2)d^2(x_0,x_1)-2(t-t^2)d^2(y_0,y_1)\\[0.2\baselineskip]
		&=d^2(x_0,y_0)+d^2(x_1,y_1)\\
		&\quad+(t+s-2st)\bigl(d^2(x_0,y_1)+d^2(x_1,y_0)-d^2(x_0,y_0)-d^2(x_1,y_1)\bigr)\\
		&\quad -2(s-s^2)d^2(x_0,x_1)-2(t-t^2)d^2(y_0,y_1).
	\end{align}
	Since $t,s\in(0,1]$ we get $s+t-2st>0$ and we can
	use~\eqref{eq:reshet} to estimate
	\begin{equation}
	\begin{split}
			d^2_{\HH^2}\bigl(\mathcal{R}_{\eta G}(x),\mathcal{R}_{\eta G}(y)\bigr)
				&\le d^2(x_0,y_0)+d^2(x_1,y_1)
					+(s+t-2st)2d(x_0,x_1)d(y_0,y_1)\\
				&\quad-2(s-s^2)d^2(x_0,x_1)-2(t-t^2)d^2(y_0,y_1).
		\end{split}\label{eq:tv_last_step}
	\end{equation}
	Plugging in the definitions of $s,t$ we conclude
	\begin{align}
		d^2_{\HH^2}&\bigl(\mathcal{R}_{\eta G}(x),\mathcal{R}_{\eta G}(y)\bigr)\\
		&\le d^2(x_0,y_0)+d^2(x_1,y_1)\\
		&\quad+2\Bigl(\frac{2\eta}{d(x_0,x_1)}+\frac{2\eta}{d(y_0,y_1)}-\frac{8\eta^2}{d(x_0,x_1)d(y_0,y_1)}\Bigr)d(x_0,x_1)d(y_0,y_1)\\
		&\quad-2\Bigl(\frac{2\eta}{d(x_0,x_1)}-\frac{4\eta^2}{d^2(x_0,x_1)}\Bigr)d^2(x_0,x_1)-2\Bigl(\frac{2\eta}{d(y_0,y_1)}-
		\frac{4\eta^2}{d^2(y_0,y_1)}\Bigr)d^2(y_0,y_1)\\[0.2\baselineskip]
		&=d_{\HH^2}^2\bigl((x_0,x_1),(y_0,y_1)\bigr).
	\end{align}
	% 2.3
	\item For $d(x_0,x_1) \le 2 \eta \le d(y_0,y_1)$ we have
	$s=1$ and $t = \tfrac{2\eta}{d(y_0,y_1)}$. 
	Substituting these values into~\eqref{eq:tv_last_step} we obtain
	\begin{align}
		d^2_{\HH^2}\bigl(\mathcal{R}_{\eta G}(x),\mathcal{R}_{\eta G}(y)\bigr)&\le d^2(x_0,y_0)+d^2(x_1,y_1)\\
		&\quad+\Bigl(1 - \frac{2\eta}{d(y_0,y_1)}\Bigr)2d(x_0,x_1)d(y_0,y_1)
		-4\eta d(y_0,y_1)+8\eta^2
	\end{align}
	and using $d(x_0,x_1) \le 2 \eta$ we obtain further
	\begin{align}
		d^2_{\HH^2}\bigl(\mathcal{R}_{\eta G}(x),\mathcal{R}_{\eta G}(y)\bigr)
		&\le d^2(x_0,y_0)+d^2(x_1,y_1)
		= d_{\HH^2}^2(x,y).
	\end{align}
	\end{enumerate} % subcases of 2
\end{enumerate}
This finishes the proof.
\end{proof}

%----------------------------------------------------------------
\section{Reflections on Convex Sets} \label{subsec:proj}
%----------------------------------------------------------------
%
Next we deal with the reflection operator ${\mathcal R}_{\iota_{\textsf{D}}}$,
i.e., the reflection operator which corresponds to the orthogonal projection
operator onto
the convex set $\textsf{D}$. Unfortunately, in symmetric Hadamard manifolds,
reflections corresponding to orthogonal projections onto convex sets are in
general not nonexpansive.
Counterexamples can be found in~\cite{BH1999,FL2013}.
Unfortunately this is also true for our special set $\textsf{D}$ as the
following example with symmetric positive definite $2 \times 2$ matrices $\SPD(2)$ shows.
For the manifold ${\mathcal P}(n)$ of symmetric positive definite matrices see Appendix~\ref{app:spd}.
\begin{example}
Let $\textsf{D} = \bigl\{(x,x,x) : x\in\SPD(2)\bigr\}$ and $\vect{x},\vect{y}\in\SPD(2)^3$ be given by
\begin{align}
\vect{x} = \Biggl(
\begin{pmatrix}
20.9943 & 3.3101 \\ 3.3101 & 6.8906
\end{pmatrix} ,\begin{pmatrix}
17.2428 & 4.3111 \\ 4.3111 & 9.9950
\end{pmatrix} ,\begin{pmatrix}
19.4800 & 19.8697 \\ 19.8697 & 21.2513
\end{pmatrix}
\Biggr),\\
\vect{y}=\Biggl(
\begin{pmatrix}
7.5521 & 6.0509 \\ 6.0509 & 19.8961
\end{pmatrix}, \begin{pmatrix}
6.4261 & 5.7573 \\ 5.7573 & 15.2775
\end{pmatrix}, \begin{pmatrix}
12.4792 & 12.9202 \\ 12.9202 & 13.8620
\end{pmatrix}
\Biggr).
\end{align}
The distance between $\vect{x}$ and $\vect{y}$ is  $d_{{\mathcal P}(2)}(\vect{x},\vect{y}) \approx 2.2856.$\\
The projection onto $D$ can be calculated using the gradient descent method from \cite{ATV13}. We obtain 
\begin{align}
 \prox_{\iota_{\textsf{D}}}(\vect{x})  \approx (1,1,1)^\tT \otimes \begin{pmatrix}
 13.8254 & 8.7522 \\ 8.7522 & 10.8436
 \end{pmatrix},\\
  \prox_{\iota_{\textsf{D}}}( \vect{y})  \approx (1,1,1)^\tT \otimes \begin{pmatrix}
  8.3908 & 8.2797 \\ 8.2797 & 12.4013
  \end{pmatrix},
\end{align}
where $\otimes$ denotes the Kronecker product.
For the distance between them we obtain 
\[d_{{\mathcal P}(2)}({\mathcal R}_{\iota_{\textsf{D}}}\vect{x},{\mathcal R}_{\iota_{\textsf{D}}}\vect{y})
\approx 2.7707 > d_{{\mathcal P}(2)}(\vect{x},\vect{y}),\]
i.e., 
the reflection is not nonexpansive. 
The computations were done with machine precision in \textsc{Matlab} and the results are rounded to four digits.
\end{example}
The situation changes if we consider manifolds with constant
curvature $\kappa$.
We do not restrict ourselves to Hadamard manifolds now, but deal instead with
the model spaces for constant curvature ${\mathcal M}^d_\kappa$ which are
defined as follows:
For $\kappa>0$ the model space ${\mathcal M}^d_\kappa$ is obtained from the $d$-dimensional
sphere by multiplying the distance with $\frac{1}{\sqrt{\kappa}}$; 
for $\kappa < 0$ we get ${\mathcal M}^d_\kappa$ from the $d$-dimensional
hyperbolic plane by multiplying the distance with $\frac{1}{\sqrt{-\kappa}}$; 
finally ${\mathcal M}^n_0$ is the $d$-dimensional Euclidean space. 
The model spaces inherit their geometrical properties from the three Riemannian
manifolds that define them. 
Thus, if $\kappa<0$, then ${\mathcal M}^d_\kappa$ is uniquely geodesic, balls
are convex and we have a counterpart for the hyperbolic law of cosine. 
By $r_{\kappa}$ we denote the convex radius of the model
spaces ${\mathcal M}^d_\kappa$ which is the supremum of radii of balls, which
are convex in ${\mathcal M}^d_\kappa$, i.e.,
\begin{equation}
r_{\kappa}\coloneqq \begin{cases}
\infty,\quad &\kappa\le0,\\
\frac{\pi}{\sqrt{k}},\quad &\text{otherwise}.
\end{cases}
\end{equation} 
To show, that reflections at convex sets in manifolds with constant curvature are nonexpansive, we need the following properties of projections onto convex sets.
\begin{proposition}\cite[Proposion II.2.4, Exercise II.2.6(1)]{BH1999}\label{prop:proj}
	
	Let X be a complete CAT$(\kappa),\ \kappa\in\RR,$ space, $V=\{x:d(x,V)\le r_\kappa/2\},\ x\in V$, and $C\subset X$ is closed and convex. Then the following statements hold
	\begin{enumerate}
		\item The metric projection $\Pi_C(x)$ of $x$ onto $C$ is a singleton.
		\item If $y\in[x,\Pi_C(x)]$, then $\Pi_C(x) = \Pi_C(y)$.
		\item If $x\notin C,\ y\in C,$ and $ y\neq \Pi_C(x)$, then $\angle_{\Pi_C(x)}(x,y)\ge\frac{\pi}{2}$, where $\angle_{\Pi_C(x)}(x,y)$ denotes the angle at $\Pi_C(x)$ between the geodesics $\gamma_{\overset{\frown}{\Pi_C(x),x}}$ and $\gamma_{\overset{\frown}{\Pi_C(x),y}}$.
	\end{enumerate}	
\end{proposition}
%
%-------------------------------------------------------------------------------
\begin{theorem}\label{thm:refl-iota-nonexp}
	Let $k\in\RR$ and $d\in\NN$. 
	Suppose that ${\mathcal C}$ is a nonempty closed and convex subset
	of ${\mathcal M}^d_\kappa$. Let $x,y\in{\mathcal M}^d_\kappa$ 
	such that
	$\operatorname{dist}(x,{\mathcal C}),\operatorname{dist}(y,{\mathcal C})\le r_\kappa$ be given.
	Then 
	\begin{equation}
	d(\mathcal{R}_{\iota_{\mathcal C}}x,\mathcal{R}_{\iota_{\mathcal C}}y)\le d(x,y).
	\end{equation}
\end{theorem}
%-------------------------------------------------------------------------------
%
\begin{proof} 
	The case $\kappa<0$ was proved in~\cite{FL2013}.
	For $\kappa = 0$ the assumption follows from the Hilbert space setting. 
	We adapt the proof from~\cite{FL2013} to show the case $\kappa>0$.
	Due to the structure of the model spaces it is sufficient to show the nonexpansiveness for $\kappa = 1$.
	
	For simplicity, we denote $c_x = \Pi_C(x),\ c_y = \Pi_C(y),\ x'=R_C(x),\ y' = R_C(y),\ \alpha= \angle_{c_x}(x,c_y),$ and $\alpha'= \angle_{c_x}(y,c_x)$. Notice, that by Proposition \ref{prop:proj}, $\alpha,\alpha' \ge \frac{\pi}{2}$, and $d(x,y)\le\pi,\forall y,x\in\mathbb{S}^n$. Consider the geodesic triangles $\triangle xc_xc_y$ and $\triangle x'c_xc_y$. By the spherical law of cosines we have that
	\begin{equation}
	\cos \bigl(d(x,c_y)\bigr) = \cos\bigl( d(x,c_x)\bigr) \cos\bigl( d(c_x,c_y)\bigr)+\sin\bigl( d(x,c_x)\bigr)\sin\bigl( d(c_x,c_y)\bigr)\cos(\alpha)
	\end{equation}
	and 
	\begin{equation}
	\cos \bigl(d(x',c_y)\bigr) = \cos\bigl( d(x',c_x)\bigr) \cos\bigl( d(c_x,c_y)\bigr)+\sin\bigl( d(x',c_x)\bigr)\sin\bigl( d(c_x,c_y)\bigr)\cos(\pi-\alpha).
	\end{equation}
	Since $d(x,c_x)=d(x',c_x)$ and $\cos(\pi-\alpha)\ge0$ we get
	\begin{equation}
	 \cos \bigl(d(x,c_y)\bigr)\le\cos \bigl(d(x',c_y)\bigr).\label{eq:reflection_to_prox_x}
	\end{equation}
	Similarly we get 
	\begin{equation}
	\cos \bigl(d(y,c_x)\bigr)\le\cos\bigl( d(y',c_x)\bigr).\label{eq:reflection_to_prox_y}
	\end{equation}	
	Consider now the geodesic triangles $\triangle x'c_xc_y$ and $\triangle x'c_xc_y$ and denote $\beta = \angle_{c_x}(x,y)$. Applying again the spherical law of cosines we obtain that
	\begin{equation}
	\cos\bigl( d(x,y)\bigr) = \cos \bigl(d(x,c_x)\bigr) \cos\bigl( d(c_x,y)\bigr)+\sin \bigl(d(x,c_x)\bigr)\sin\bigl( d(c_x,y)\bigr)\cos(\beta)
	\end{equation}
	and
	\begin{equation}
	\cos \bigl(d(x',y)\bigr) = \cos \bigl(d(x',c_x)\bigr) \cos \bigl(d(c_x,y)\bigr)+\sin \bigl(d(x',c_x)\bigr)\sin \bigl(d(c_x,y)\bigr)\cos(\pi-\beta).
	\end{equation}
	Since $d(x,c_x)=d(x',c_x)$ and $\cos(\pi-\beta)= -\cos\beta$, we get by adding
	\begin{align}
	\cos \bigl(d(x,y)\bigr)+\cos\bigl( d(x',y) \bigr)= 2\cos \bigl(d(x,c_x)\bigr) \cos \bigl(d(c_x,y)\bigr)\label{eq:cos_contra1}
	\end{align}
	and similar
	\begin{align}
	\cos \bigl(d(x',y')\bigr)+\cos\bigl( d(x,y')\bigr) &= 2\cos \bigl(d(x,c_x)\bigr) \cos\bigl( d(c_x,y')\bigr)\label{eq:cos_contra2},\\
	\cos\bigl( d(x,y)\bigr)+\cos \bigl(d(x,y')\bigr) &= 2\cos \bigl(d(x,c_y)\bigr) \cos\bigl( d(c_y,y)\bigr),\label{eq:cos_contra3}\\
	\cos\bigl( d(x',y')\bigr)+\cos\bigl( d(x',y)\bigr) &= 2\cos\bigl( d(x',c_y)\bigr) \cos\bigl( d(c_y,y)\bigr).\label{eq:cos_contra4}	
	\end{align}
	Suppose now that $d(x',y')>d(x,y)$. From \eqref{eq:cos_contra1}, \eqref{eq:reflection_to_prox_y}, and \eqref{eq:cos_contra2} we obtain
	\begin{align}
	\cos\bigl( d(x',y')\bigr)+\cos\bigl( d(x,y')\bigr) &= 2\cos \bigl(d(x,c_x) \bigr)\cos\bigl( d(c_x,y')\bigr)\\
	&\ge 2\cos\bigl( d(x,c_x) \bigr)\cos\bigl( d(c_x,y)\bigr)\\
	&=\cos \bigl(d(x,y)\bigr)+\cos\bigl( d(x',y)\bigr)\\
	&>\cos \bigl(d(x',y')\bigr)+\cos\bigl( d(x',y)\bigr)
	\end{align}
	which implies $d(x,y')<d(x',y)$. Using  \eqref{eq:cos_contra3}, \eqref{eq:reflection_to_prox_x}, and \eqref{eq:cos_contra4} we get
	\begin{align}
	\cos\bigl(d(x,y)\bigr)+\cos \bigl(d(x,y')\bigr) &= 2\cos\bigl(d(x,c_y)\bigr) \cos\bigl( d(c_y,y)\bigr)\\
	&\le 2\cos\bigl( d(x',c_y)\bigr) \cos \bigl(d(c_y,y)\bigr)\\
	&=\cos\bigl( d(x',y')\bigr)+\cos\bigl( d(x',y)\bigr)\\	
	&<\cos\bigl( d(x,y)\bigr)+\cos\bigl( d(x',y)\bigr)
	\end{align}
	which is a contradiction to $d(x,y')<d(x',y)$, thus the result follows.
\end{proof}
		
%-------------------------------------------------------------------------------
%
%-------------------------------------------------------------------------------
\section{Numerical Examples} \label{sec:numerics}
%-------------------------------------------------------------------------------
In this section, we demonstrate the performance of the Parallel DR Algorithm \ref{alg:DR_par_had} (PDRA)
by numerical examples. We chose a constant step size $\lambda_r\coloneqq \lambda \in (0,1)$ in the reflections
for all $r \in \mathbb N$, where the specific value of $\lambda$ is stated in each experiment.

We compare our algorithm with two other approaches to minimize the functional ${\mathcal E}$ in~\eqref{task}, 
namely 
the Cyclic Proximal Point Algorithm (CPPA)~\cite{BBSW2015,WDS2014} 
and the 
Half-Quadratic Minimization Algorithm in its multiplicative form (HQMA)~\cite{BCHPS15}. 
Let us briefly recall these algorithms.
\begin{description}
 \item[CPPA.] Based on the splitting \eqref{task_sum_TV} of our functional~\eqref{task} 
 the CPPA computes the iterates
 \[
 u^{(r+1)} = \prox_{\eta_r{\varphi}_5} \circ \ldots \circ \prox_{\eta_r{\varphi}_1} \bigl( u^{(r)} \bigr).
 \]
 The algorithm converges for any starting point $u^{(0)}$ to a minimizer of the functional~\eqref{task} 
 supposed that
 \[
 \sum_{r \in \mathbb N} \eta_r = +\infty, \quad  \sum_{r \in \mathbb N} \eta_r^2 < +\infty.
 \]
 Therefore the sequence  $\{\eta_r\}_r$ must decrease.
 In our computations we chose $\eta_r\coloneqq \frac{\eta}{r+1}$, where the specific value of $\eta$ is addressed in the experiments.
 
  \item[HQMA.] The HQMA cannot be directly applied to ${\mathcal E}$, but to its
  differentiable substitute
  \begin{equation}
{\mathcal E}_\delta (u) \coloneqq\frac{1}{2}
\sum_{(i,j) \in {\mathcal V}} d (f_{i,j},u_{i,j})^2
+ \alpha %\underbrace{
\biggl(
\sum_{(i,j) \in {\mathcal G}} \varphi\bigl(d (u_{i,j},u_{i+1,j}) \bigr)
+\sum_{(i,j) \in {\mathcal G}} \varphi\bigl(d (u_{i,j},u_{i,j+1})\bigr)
\biggr),
\end{equation}
where $\varphi(t)\coloneqq \sqrt{t^2+\delta^2}$, $\delta > 0$. For small $\delta$ the functional ${\mathcal E}_\delta$ approximates  ${\mathcal E}$.
Other choices of $\varphi$ are possible, see, e.g., \cite{BCHPS15,NN2005}.
The functional can be rewritten as
\begin{align}
\begin{split}
{\mathcal E}_\delta (u)  
=&\min_{\vect{v},\vect{w}\in\RR^{N,M}}\Biggl\{\frac{1}{2}
\sum_{(i,j) \in {\mathcal V}} d (f_{i,j},u_{i,j})^2\\
&+ \alpha
\biggl(
\sum_{(i,j) \in {\mathcal G}}v_{i,j}d^2 (u_{i,j},u_{i+1,j}) +\psi(v_{i,j})+ w_{i,j}d^2 (u_{i,j},u_{i,j+1})+\psi(w_{i,j})
\biggr)\Biggr\}
\end{split}\\
=&\min_{\vect{v},\vect{w}\in\RR^{N,M}} E(u,\vect{v},\vect{w}),
\end{align}
where $\psi(s)\coloneqq \min_{s\in\RR}\bigl\{t^2s-\varphi(t)\bigr\}$, see~\cite{BCHPS15}. 
Then an alternating algorithm is used to minimize the right-hand side, i.e., the HQMA computes
\begin{align}
\bigl(\vect{v}^{(r)},\vect{w}^{(r)}\bigr) &= \argmin_{\vect{v},\vect{w}\in\RR^{N,M}} \operatorname{E}(u^{(r)},\vect{v},\vect{w}),\label{eq:minv_1}\\
u^{(r+1)} &= \argmin_{u\in\mathcal{H}^{N,M}} \operatorname{E}(u,\vect{v}^{(r)},\vect{w}^{(r)}).\label{eq:minu}
\end{align}
For \eqref{eq:minv_1} there exists an analytic expression, while the minimization in \eqref{eq:minu} is done with a Riemannian Newton method \cite{AMS08}. 
The whole HQMA can be considered as a quasi-Newton algorithm to minimize ${\mathcal E}_\delta$.
The convergence of $\{u^{(r)} \}_r$ to a minimizer of ${\mathcal E}_\delta$ is ensured by  \cite[Theorem 3.4]{BCHPS15}.
\end{description}

In order to compare the algorithms, we implemented them in a common framework in \textsc{Matlab}. 
They rely on the same manifold functions, like the distance function, the proximal mappings, and the exponential and logarithmic map.
These basic ingredients as well as the Riemannian Newton step in the HQMA, especially the Hessian needed therein, 
were implemented in C++ and imported into \textsc{Matlab} using \lstinline!mex!-interfaces while for all algorithms
the main iteration was implemented in \textsc{Matlab} for convenience.
The implementation is based on the \lstinline!Eigen! library%
\footnote{open source, available at
\href{http://eigen.tuxfamily.org}{http://eigen.tuxfamily.org}}
3.2.4.
All tests were run on a MacBook Pro running Mac OS X 10.11.1, Core i5, 2.6 GHz,
with 8 GB RAM using \textsc{Matlab} 2015b and the clang-700.1.76 compiler.

As \emph{initialization} of the iteration we use $u^{(0)}=f$ if all pixels of the initial image $f$ are known
and a nearest neighbor initialization for the missing pixels otherwise.
As \emph{stopping criterion} we employ  
\(\epsilon^{(r)}\coloneqq \dist\bigl(x^{(r)},x^{(r-1)}\bigr) < \epsilon\) often
in combination with a maximal number of permitted iteration steps $r_{\max}$.

We consider three manifolds: 
\begin{enumerate}[label={\arabic*.}]
 \item The space of univariate Gaussian probability distributions with the Fisher metric: 
 As in \cite{AV2014},  we associate a univariate Gaussian probability distribution to each image pixel in Subsection \ref{sec:gaussian}.
 Then we denoise this image by our approach.
 
 \item Symmetric positive definite $2\times 2$ matrices of determinant 1 with the affine invariant Riemannian metric:
Following an idea of \cite{CF2009}, we model structure tensors as elements in this space
and denoise the tensor images in Subsection~\ref{subsec:numerics:SSPD}.

 \item Symmetric positive definite $3 \times 3$ matrices with the affine invariant Riemannian metric:
 We are interested in denoising diffusion tensors in DT-MRI in Subsection~\ref{subsec:numerics:SPD}.
 While the first two experiments deal only with the denoising of images, we further perform a combined inpainting and denoising approach 
 in Subsection~\ref{subsec:numerics:InpaintSPD}.
\end{enumerate}
The first two settings are manifolds with constant curvature $-\frac12$.
Indeed, Appendix \ref{app:hn} shows that there are isomorphisms between these spaces
and the hyperbolic manifold $\Hn_{\mathrm{M}}$. Our implementations 
use  these isomorphisms to work finally on the hyperbolic manifold $\Hn_{\mathrm{M}}$.
The symmetric positive definite matrices in the third setting form a symmetric Hadamard manifold,
but its curvature is not constant. Although the reflection ${\mathcal R}_{\iota_D}$ is in general not nonexpansive
which is required in the convergence proof of Theorem \ref{th:kras}, we observed convergence in all
our numerical examples.

%-------------------------------------------------------------------------------
\subsection{Univariate Gaussian Distributions}\label{sec:gaussian}
%-------------------------------------------------------------------------------
\begin{figure}[tbp]
	%\pgfkeys{/pgf/number format/.cd,std,fixed zerofill,precision={2}}
	\centering
	\begin{subfigure}{0.3\textwidth}
		\centering
		\includegraphics{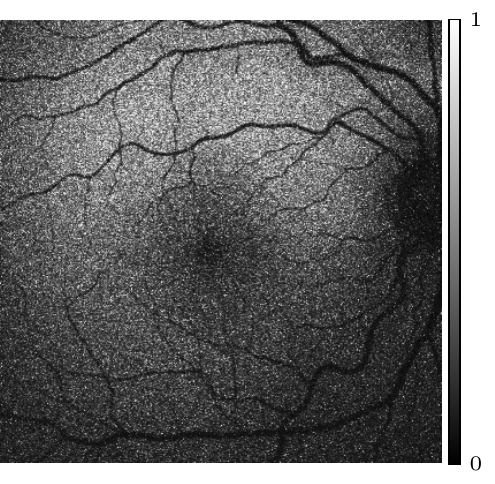}
		\vspace{-0.5\baselineskip}
		\caption[]{First image $g^{(1)}$}
	\end{subfigure}
	\hspace{3em}
	\begin{subfigure}{0.3\textwidth}
		\centering
		\includegraphics{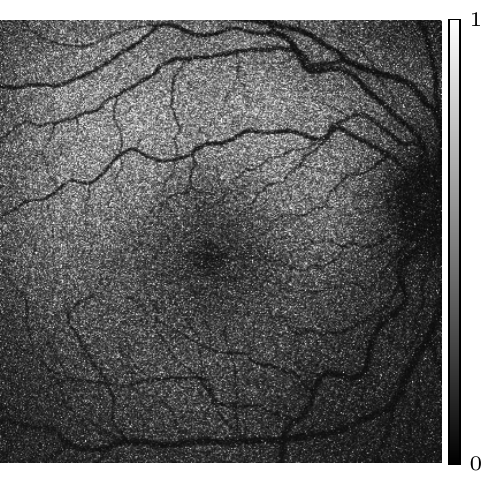}
		\vspace{-0.5\baselineskip}
		\caption[]{Last image $g^{(20)}$}
	\end{subfigure}\\
	\begin{subfigure}[t]{0.3\textwidth}
		\centering
		\includegraphics{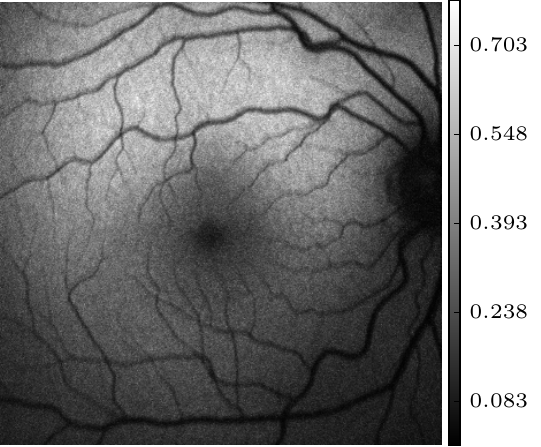}
		%\vspace{-.5\baselineskip}
		\caption[]{Original mean}
	\end{subfigure}
	\hspace{3em}
	\begin{subfigure}[t]{0.3\textwidth}
		\centering
		\includegraphics{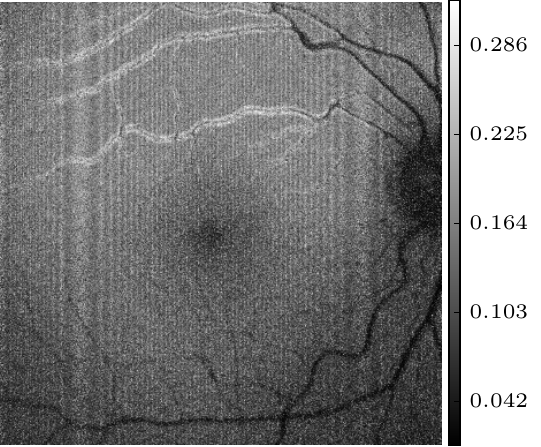}
		%\vspace{-.5\baselineskip}
		\caption[]{Original standard deviation}
	\end{subfigure}	
	\\
	\begin{subfigure}[t]{0.3\textwidth}
		\centering
		\includegraphics{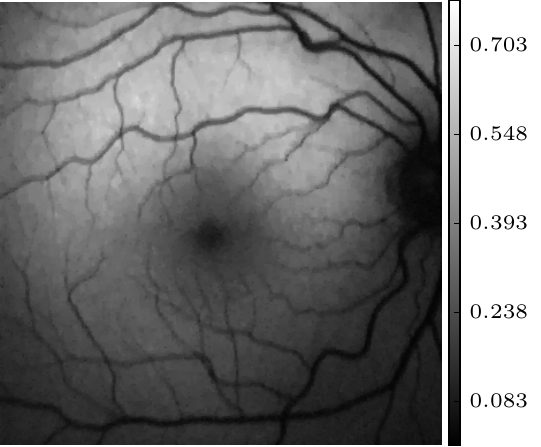}
		%\vspace{-.5\baselineskip}
		\caption[]{Restored mean}
	\end{subfigure}
	\hspace{3em}
	\begin{subfigure}[t]{0.3\textwidth}
		\centering
		\includegraphics{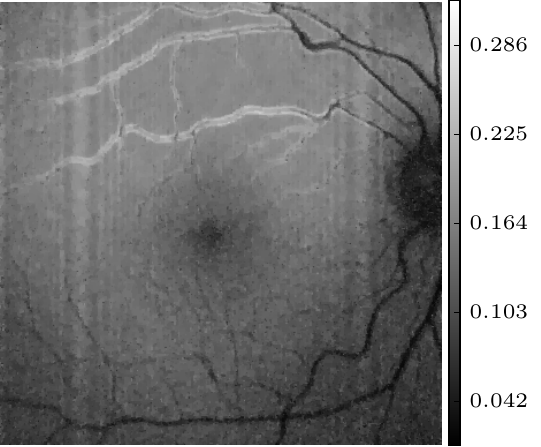}
	%	\vspace{-.5\baselineskip}
		\caption[]{Restored standard deviation}
	\end{subfigure}
	\caption[]{Denoising of the retina data using model \eqref{task} with \(\alpha=0.2\) and the PDRA ($\eta = \tfrac{1}{2}$, $\lambda =\tfrac{9}{10}$). 
	The restored image keeps the main features, e.g. veins in the mean and their movement
	in the variance in an area around them.}\label{fig:RetinaPictures}
\end{figure}

In this section, we deal with a series of $m$ similar gray-value images
$\big( g^{k}_{i,j} \big)_{(i,j) \in {\mathcal G}}$, $k=1,\ldots,m$.
We assume that the gray-values $g^{k}_{i,j}$, $k=1,\ldots,m$, of each pixel $(i,j) \in {\mathcal G}$
are realizations of a univariate Gaussian random variable with distribution $N(\mu_{i,j},\sigma_{i,j})$. 
We estimate the corresponding mean and the standard deviation
 by the maximum
likelihood estimators
\begin{equation} \label{ML}
\mu_{i,j} = \frac{1}{m}\sum_{k=1}^m g^{k}_{i,j},\quad
\sigma_{i,j} =\sqrt{ \frac{1}{m}\sum_{k=1}^m\bigl(g^{k}_{i,j}-\mu_{i,j}\bigr)^2}.
\end{equation}
We consider images $f\colon \mathcal{G}\rightarrow\mathcal{N}$ mapping into the Riemannian
manifold $\mathcal{N}$ of univariate non-degenerate Gaussian probability distributions, parameterized by
the mean and the standard deviation, with the Fisher metric.
For the definition of the Fisher metric as well as its relation to the hyperbolic manifold,
see Appendix \ref{app:hn}.

In our numerical examples we use a sequence of \( m=20 \) images $\big( g^{k}_{i,j} \big)_{i,j = 1}^{384}$ of
size \(384\times384\) taken from the same retina by a CCD (coupled charged-device) camera,
cf.~\cite[Fig.~13]{AV2014}, 
over a very short time frame, each with a very short exposure time.
Hence, we have noisy images that are also affected by the movement of the eye.
In the top row of Fig. \ref{fig:RetinaPictures} we see the first and the last image of this sequence.
From these images, we obtain the ---still noisy--- image $f\colon \mathcal{G}\rightarrow\mathcal{N}$ with
\(
\big( f_{i,j} \big)_{i,j=1}^{384} = \big( (\mu_{i,j},\sigma_{i,j}) \big)_{i,j=1}^{384} 
\)
depicted in the middle row of Fig. \ref{fig:RetinaPictures} by \eqref{ML}.
We denoise this image by minimizing the functional \eqref{task} with $\alpha = 0.2$ by the
PDRA ($\eta = \tfrac{1}{2}$, $\lambda =\tfrac{9}{10}$).
The result is shown in the bottom row of Fig. \ref{fig:RetinaPictures}.

\begin{table}
	\centering
	\begin{tabular}{lcccc}
		\toprule
		\multirow{2}{*}{$\eta$} & \multirow{2}{*}{CPPA} & \multicolumn{3}{c}{PDRA}\\
		\cmidrule{3-5}
		\multicolumn{2}{r}{\scriptsize{in sec.}} & $\lambda = 0.5$ & $\lambda = 0.9$ & $\lambda = 0.95$\\
		\midrule
$0.05$ & $56.85$ & $129.26$ & $65.21$  & $59.84$\\
 $0.1$ & $56.54$ & $59.21$  & $34.32$ & $36.67$\\
 $0.5$ & $65.17$ & $57.41$  & $42.06$  & $46.07$\\
 $1$   & $57.14$ & $93.75$  & $63.58$  & $58.66$\\		\bottomrule
	\end{tabular}
	\caption[]{Computation times for the CPPA and PDRA for different values of \(\eta\)
	and \(\lambda\) in seconds. 
	%The emphasized ones are also investigates during iterations in Fig.~\ref{fig:Retina:Iterations}.
	}
	\label{tab:times}
\end{table}

Next we compare the performance of the PDRA with the CPPA
for different values 
\(\eta\) of the proximal mapping and various parameters \(\lambda\) of the reflection, both applied to a subset of $50\times 50$ pixels of the Retina dataset.
The stopping criterion is 
\(r_{\text{max}} = 1500\) iterations or \(\epsilon = 10^{-6}\). 

Table~\ref{tab:times} records the computational time of both algorithms.
The CPPA requires always \(1500\) iterations which takes roughly a minute.
It appears to be robust with respect to the choice of \(\eta\). 
The PDRA on the other
hand depends on the chosen value of \(\eta\), e.g., for \(\lambda=0.5\) the computation time for a small \(\eta\) is \(129\) seconds, 
while its minimal computation time is just \(34\) seconds, for \(\lambda=0.9\) and \(\eta=0.1\). For the latter parameters, 
the iteration stops after \(278\) iterations with an \(\epsilon^{(278)} < 10^{-6}\). 
The runtime per iteration is longer  than for the CPPA due to the Karcher mean computation 
in each iteration which is implemented using the gradient descent method of~\cite{ATV13}.

\begin{table}
	\centering
	\begin{tabular}{lcccc}
		\toprule
		\multirow{2}{*}{$\eta$} & \multirow{2}{*}{CPPA} & \multicolumn{3}{c}{PDRA}\\
		\cmidrule{3-5}
		\multicolumn{2}{r}{\scriptsize{\(184.3643\)+...}} & $\lambda = 0.5$ & $\lambda = 0.9$ & $\lambda = 0.95$\\
		\midrule
$0.05$ & $44.80$ & $1.021\times10^{-5}$ & $1.180\times10^{-5}$  & $1.627\times10^{-5}$\\
 $0.1$ & $10.65$ & $2.514\times10^{-5}$ & $2.969\times10^{-5}$  & $3.429\times10^{-5}$\\
 $0.5$ & $1.055\times10^{-2}$ & $5.082\times10^{-4}$ & $2.785\times10^{-4}$  & $2.256\times10^{-4}$\\
 $1$ & $1.953\times10^{-2}$ & $8.189\times10^{-4}$ & $5.027\times10^{-4}$  & $4.992\times10^{-4}$\\ 	\bottomrule
	\end{tabular}
		\caption[]{Distance of the computed minimum to the assumed one of 184.3643.}
		%of the \(\ell^2\)-\(\TV\) functional for the same values as in Table~\ref{tab:times}. The emphasized ones are also
		%investigates during iterations in Fig.~\ref{fig:Retina:Iterations}.}
		\label{tab:l2tv}
\end{table}

While the computation time does not change much for the CPPA with different $\eta$, the resulting
values of the functional ${\mathcal E}$ in \eqref{task} do significantly. 
We compare the same values for \(\eta\) and \(\lambda\) as for the time measurements in Table~\ref{tab:l2tv}. 
To this end, we estimated the minimum of the functional as \(184.3643\) by performing the PDRA with \(3000\) iterations. 
We use this value to compare the results obtained for different parameters with this estimated minimum. 
Though the fastest PDRA $(\lambda = 0.9,\ \eta=0.1)$ does not yield the lowest value of ${\mathcal E}$, 
the slowest one $(\lambda = 0.5,\ \eta=0.05)$ does. 
Looking at the values for the CPPA, 
they depend heavily on the parameter \(\eta\) and are further away from the minimum than all PDRA tests.

\begin{figure}\centering
	\begin{subfigure}{.49\textwidth}
		\includegraphics{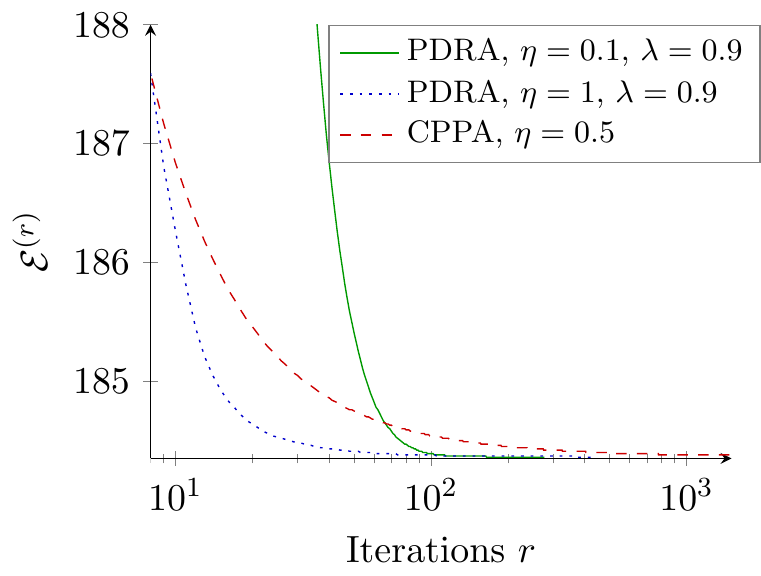}
			\caption[]{Functional value \({\mathcal E}^{(r)}\)  }
			\label{subfig:plotl2TV}
		\end{subfigure}%
	\begin{subfigure}{.49\textwidth}\centering
		\includegraphics{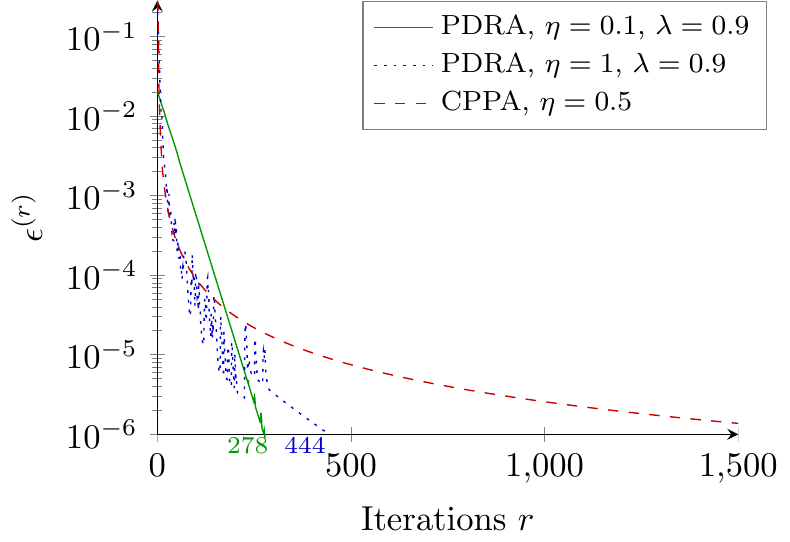}
	\caption[]{Error $\epsilon^{(r)} $}
	\label{subfig:ploteps}
	\end{subfigure}
	\caption[]{Comparison of \subref{subfig:plotl2TV} the functional value \({\mathcal E}^{(r)}\)
	and \subref{subfig:ploteps} the error \(\epsilon^{(r)}\)
	for the CPPA and the PDRA for different  parameters. 
	%While the CPPA behaves similar for different parameters \(\eta\), the DR algorithm depends on their
	%choice. 
	The  PDRA converges faster both with respect to the functional value 
	and the error.}
	\label{fig:Retina:Iterations}
\end{figure}

Fig.~\ref{subfig:plotl2TV} left shows the development of  \({\mathcal E}^{(r)}\coloneqq{\mathcal E} (u^{(r)}) \) with respect to the number of iterations. 
The CPPA converges quite slowly. Both plotted PDRA tests, the fastest (solid green) and the one starting with the steepest descent (dashed blue) from Table \ref{tab:l2tv}, 
decay faster. Due to the fast decay of $\mathcal{E}^{(r)}$ at the beginning, the iteration number is shown in a logarithmic scale.
Furthermore, also the errors \(\epsilon^{(r)}\coloneqq  \dist\bigl(x^{(r)},x^{(r-1)}\bigr)\), which are plotted in a logarithmic scale in Fig.~\ref{subfig:ploteps}, 
decay much faster for the PDRA than for the CPPA. Nevertheless, for $\eta = 1$, 
the decay is not monotone at the beginning of the algorithm.

In summary, the PDRA performs better than the CPPA with respect to runtime and minimal functional value. 
It is nearly independent of the choice of \(\eta\) when looking at the functional values in contrast to the CPPA. 
However, both \(\eta\) and \(\lambda\) have an influence on the runtime of the algorithm. 

%------------------------------------------------------------------------
\subsection{Scaled Structure Tensor} \label{subsec:numerics:SSPD}
%------------------------------------------------------------------------
The structure tensor of F\"orstner and G\"ulch~\cite{FG1987}, see also \cite{We97}, can be used to
determine edges and corners in images. For each pixel $x \in {\mathcal G}$ of an image $f\colon{\mathcal G} \rightarrow \mathbb R$ 
it is defined 
as the $2 \times 2$ matrix
\[
{\mathcal J}_\rho(x)\coloneqq \bigl(G_\rho* \nabla f_\sigma  \nabla f_\sigma ^\tT\bigr)(x) = \sum_{y \in {\mathcal G}} G_\rho(x-y) \nabla f_\sigma (y) \nabla f_\sigma (y)^\tT  .
\]
Here 
\begin{itemize}
 \item
$f_\sigma$ is the convolution of the initial image $f$ with a discretized Gaussian of zero mean and small standard deviation $\sigma$
truncated between $[-3\sigma,3\sigma]$ and mirrored at the boundary;
this slight smoothing of the image before taking the discrete derivatives avoids too noisy gradients,
\item
$\nabla$ denotes a discrete gradient operator, e.g., in this paper, forward differences in vertical and horizontal directions
with mirror boundary condition, and 
\item 
the convolution of the rank-1 matrices $\nabla f_\sigma (y) \nabla f_\sigma (y)^\tT$ with the discrete Gaussian $G_\rho$
of zero mean and standard deviation $\rho$ truncated between $[-3\rho,3\rho]$ is performed for every coefficient of the matrix;
the choice of the parameter $\rho$ usually results in
a tradeoff between
keeping too much noise in the image for small values and
blurring the edges for large values.
\end{itemize}
Clearly, ${\mathcal J}_\rho(x)$ is a symmetric, positive semidefinite matrix. For natural or  noisy images $f$ it is in general
positive definite, i.e. an element of ${\mathcal P}(2)$.
Otherwise one could also exclude matrices with zero eigenvalues by choosing~\(\mathcal V\subset\mathcal G\) appropriately in our model~\eqref{task}.

Let $\lambda_1 \ge \lambda_2 >0$ denote the eigenvalues of ${\mathcal J}_\rho(x)$, $x \in {\mathcal G}$,
with corresponding normed eigenvectors $v_1$ and $v_2 = v_1^\perp$. Then
$\lambda_1 \gg \lambda_2$, i.e., $\frac{\lambda_1}{\lambda_2} \gg 1$ indicates an edge at $x\in {\mathcal G}$.
If the quotient is near 1 we have a homogeneous neighborhood of $x$ or a vertex.
In this paper, we are interested in edges. 
To this end, we follow an approach from \cite{CF2009} and consider
\[
J_\rho(x)\coloneqq \frac{1}{\sqrt{ \det {\mathcal J}_\rho(x) }} {\mathcal J}_\rho(x).
\]
Clearly, this matrix has the same eigenvectors as ${\mathcal J}_\rho(x)$ and the eigenvalues $\mu_1 = \sqrt{\frac{\lambda_1}{\lambda_2}}$ and $\mu_2 = 1/\mu_1$.
Further, $\mu_1 \gg \mu_2$ indicates an edge.
The matrix belongs to the subspace ${\mathcal P}_1(2) \subset {\mathcal P}(2)$ of symmetric positive definite matrices
of determinant 1. 
The set ${\mathcal P}_1(2)$ together with the affine invariant metric of ${\mathcal P}(2)$ 
forms a Riemannian manifold which is isomorphic to the hyperbolic manifold
of constant curvature $-\frac12$, see Appendix \ref{app:hn}.

We consider images $\big( J_\rho(x) \big)_{x \in {\mathcal G}}$ with values in ${\mathcal P}_1(2)$.
The structure tensors are visualized as ellipses using
the scaled eigenvectors \(\mu_i v_i\), $i=1,2$ as major axes.
The colorization is done with the anisotropy index from~\cite{MoBa06}, i.e.,
tensors indicating edges are purple or blue, while red or green ellipses
indicate constant regions.

We take the artificial image of size \(N=M=64\) with values in $[0,1]$ shown in Fig.~\ref{fig:str:oi}
and add white Gaussian noise with standard deviation \(\sigma_{\text{n}} = 0.2\)
to the image.
The noisy image is depicted in Fig.~\ref{fig:str:ni}.
Fig. \ref{fig:str:sp} shows the structure tensor image $J_{\rho_1}$ for $\sigma = 0.8$ and small
$\rho_1= \frac{35}{100}$.
On the one hand, the structure tensors nicely show sharp edges,
but on the other hand 
they are affected by a huge
amount of noise.

Searching for a value of the parameter \(\rho_2\) on a grid of size \(\frac{1}{20}\) in order to
find a structure tensors with thinnest edges and constant regions results
in \(\rho_2=\frac{6}{5}\) and the structure tensors depicted in Fig.~\ref{fig:str:op}. 
The larger parameter $\rho$ increases the size of the neighborhood
in the smoothing of the matrices 
$\nabla f_\sigma  \nabla f_\sigma ^\tT$. 
This yields smoother structure tensors, but the edges become
broader (blurred). Some artifacts from the noise are still visible.

A remedy is shown in Fig.~\ref{fig:str:dr}, where 
we denoised the structure tensor image from Fig.~\ref{fig:str:sp}
using the PDRA applied to model~\eqref{task} with the parameters \(\alpha=1.05\), \(\eta=0.4\), and \(\lambda = 0.9\). 
The denoised structure tensor image shows thinner edges and less noise than the one in Fig.~\ref{fig:str:op}.

Note that the manifold-valued computations in the PDRA 
were done with respect to the hyperbolic manifold with pre-processing (post-processing) 
by the isomorphism (inverse isomorphism)  from ${\mathcal P}_1(2)$ to the hyperbolic manifold described in Appendix~\ref{app:hn}.

\begin{figure}[tbp]
	\centering
	\hfill
	\begin{subfigure}{0.27\textwidth}
		\centering
		\includegraphics[width = 0.98\textwidth]{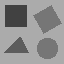}
		\caption[]{Original image}\label{fig:str:oi}
	\end{subfigure}
	\hfill
	\begin{subfigure}{0.27\textwidth}
		\centering
		\includegraphics[width = 0.98\textwidth]{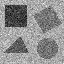}
		\caption[]{Noisy image, $\sigma_{\text{n}} = 0.2$}\label{fig:str:ni}
	\end{subfigure}
	\hfill~
	\\
	\begin{subfigure}{0.27\textwidth}
		\centering
		\ \\[-\baselineskip]
		\includegraphics[width=1.43\textwidth]{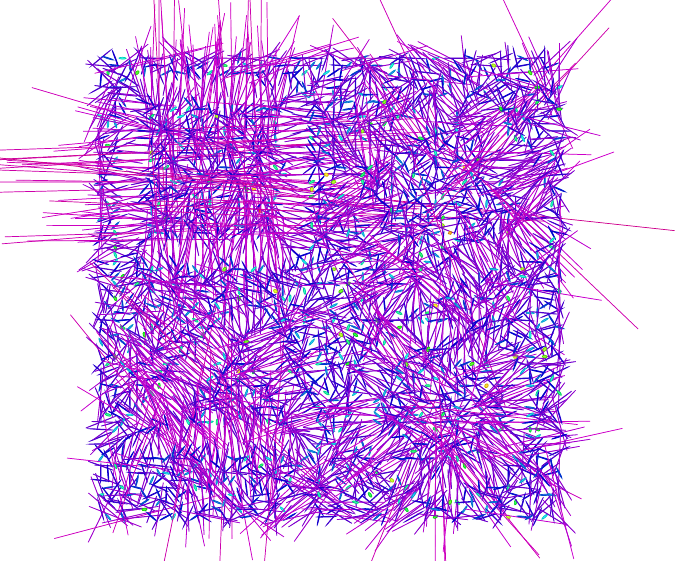}
		\\[-.5\baselineskip]
		\caption[]{Structure tensors of~\subref{fig:str:ni},
		$\sigma=\tfrac{4}{5},\rho_1=\tfrac{35}{100}$}
		\label{fig:str:sp}
	\end{subfigure}
	\hfill
	\begin{subfigure}{0.27\textwidth}
		\centering
		\includegraphics[width = \textwidth]{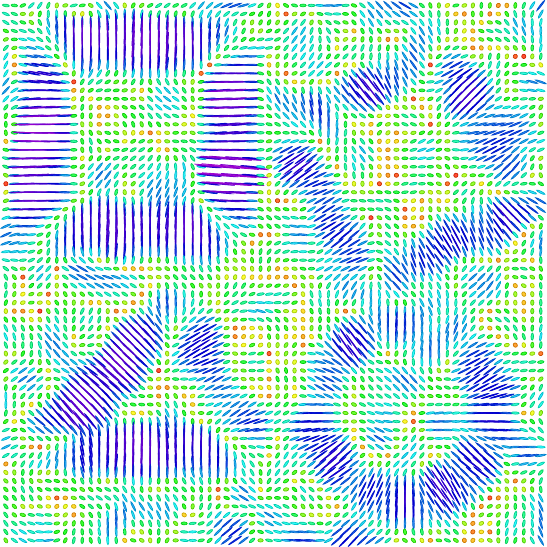}
		\caption[]{Structure tensors of~\subref{fig:str:ni},
		$\sigma=\frac{4}{5},\rho_2=\tfrac{6}{5}$.}\label{fig:str:op}
	\end{subfigure}
	\hspace{.5cm}
	\begin{subfigure}{0.27\textwidth}
		\centering
		\includegraphics[width = \textwidth]{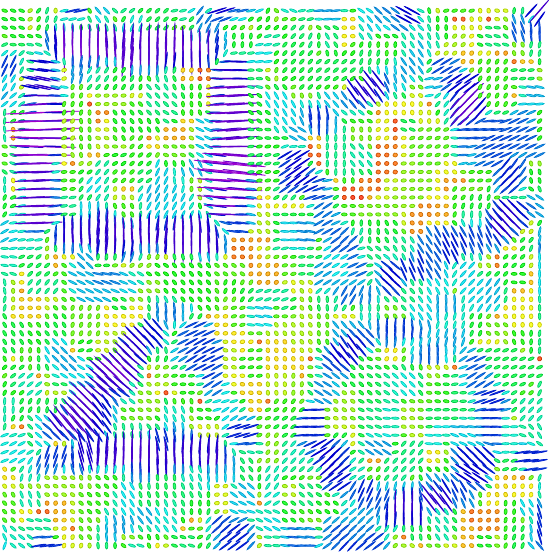}
		\caption[]{Restored structure tensors from~\subref{fig:str:sp}}\label{fig:str:dr}
	\end{subfigure}
	\caption[]{Application of model \eqref{task} and the PDRA to ${\mathcal P}_1(2)$-valued images.}
\end{figure}

%------------------------------------------------------------------------
\subsection{Denoising of Symmetric Positive Definite Matrices: DT-MRI} \label{subsec:numerics:SPD}
%------------------------------------------------------------------------
\begin{figure}\centering
	\begin{subfigure}{.495\textwidth}\centering
		\includegraphics[width=.9\textwidth]{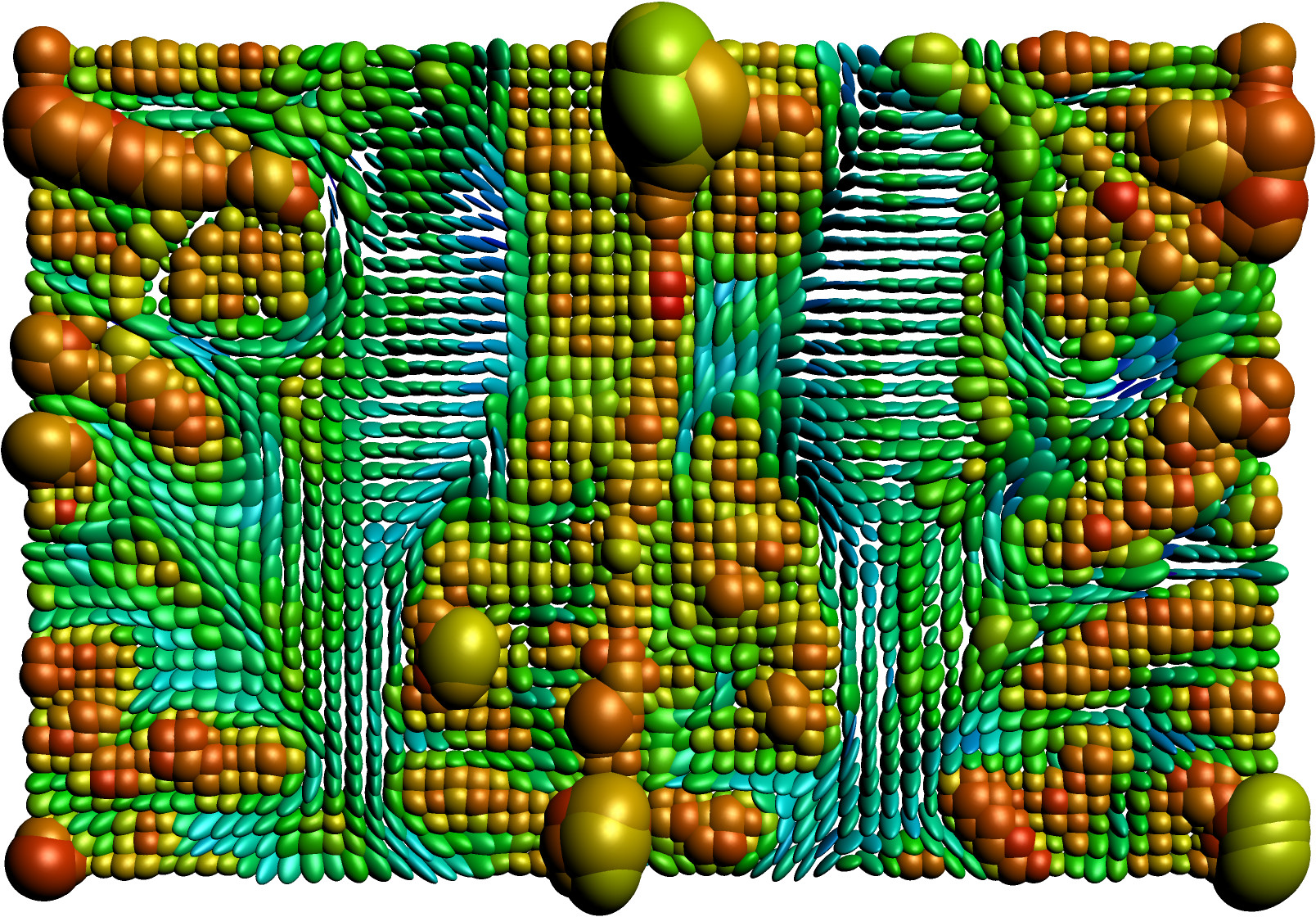}
		\caption[]{Original data from the Camino set}%
		\label{subfig:Camino:orig}
	\end{subfigure}
	\begin{subfigure}{.495\textwidth}\centering
		\includegraphics[width=.9\textwidth]{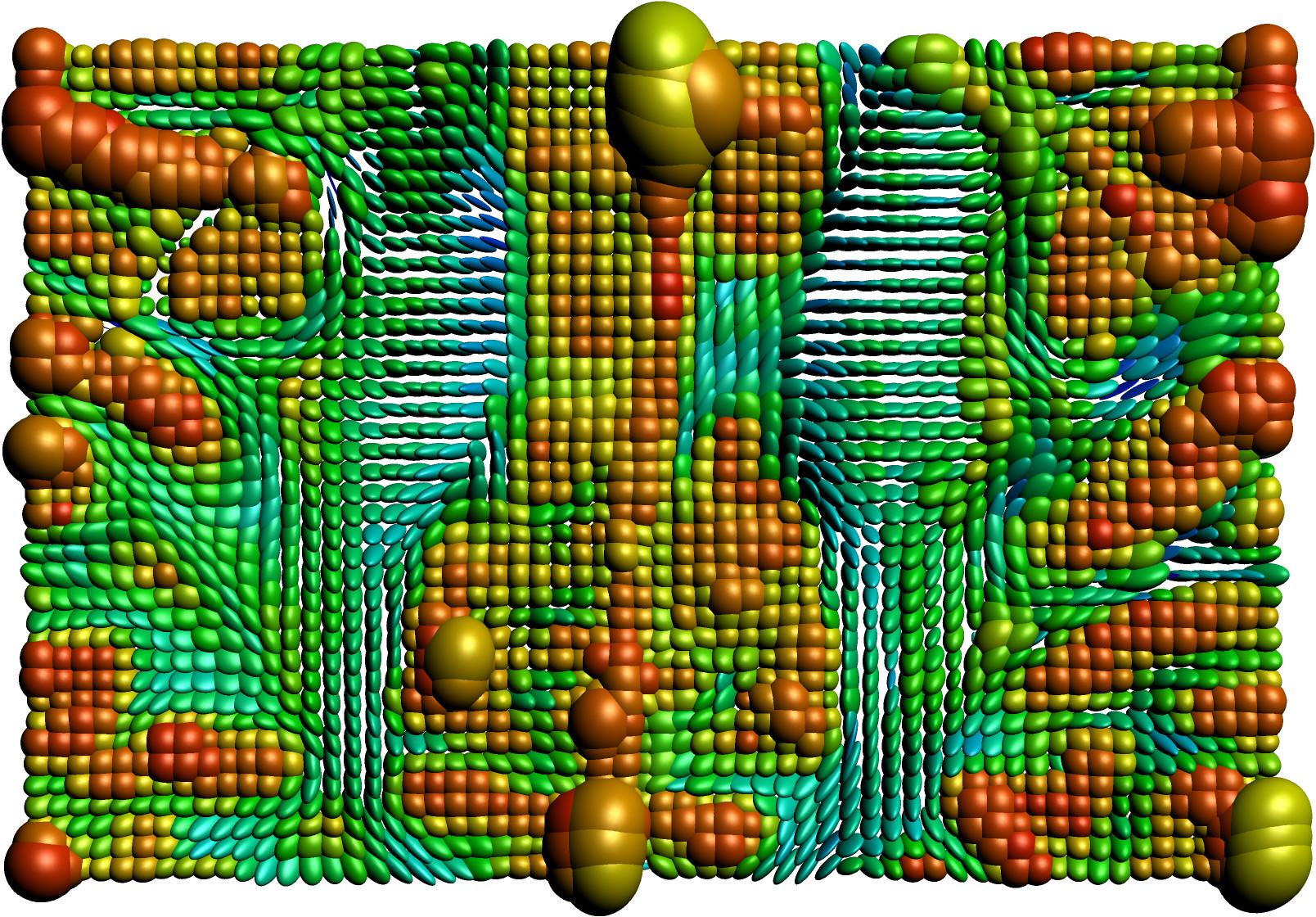}
		\caption[]{Reconstruction by model \eqref{task}, \(\alpha=0.05\)}%
		\label{subfig:Camino:reg}
	\end{subfigure}
	\caption[]{The Camino DT-MRI data of slice \(28\) and traversal plane subset
		\((i,j)\in\{28,...,87\}\times\{24,\ldots,73\}\) (left) and its denoised version by
		model \eqref{task} with the PDRA.
	}
	\label{fig:Camino}
\end{figure}

In magnetic resonance tomography, it is possible to capture the diffusivity
of the measured material and obtain diffusion tensor images (DT-MRI), where each
pixel is a \(3\times 3\) symmetric positive definite matrix. 
As already mentioned, the manifold \(\mathcal P(3)\) is not of constant curvature,
so that we cannot ensure the convergence of the PDRA in general.
Nevertheless, we can investigate the convergence numerically. 
We compare the  algorithm with the CPPA and HQMA.
For the HQMA we chose two different parameters \(\delta_1\coloneqq\tfrac{1}{10}\) and \(\delta_2\coloneqq\frac{1}{100}\) in the smoothing function $\varphi$.

We process the Camino
dataset\footnote{see~\href{http://cmic.cs.ucl.ac.uk/camino/}{http://cmic.cs.ucl.ac.uk/camino}}\cite{Camino}
which captures the diffusion inside a human head. 
To be precise, we take a subset from the complete
dataset \(f = \bigl(f_{i,j,k}\bigr)\in\mathcal P(3)^{112\times112\times50}\), 
namely from the traversal plane \(k=28\) the data points $\bigl(f_{i,j,28}\bigr)$ with \((i,j)\in\{28,...,87\}\times\{24,\ldots,73\}\).
In the original Camino data some of the pixel are
missing which we reset using a
nearest neighbor approximation depicted in Fig.~\ref{subfig:Camino:reg}. 
The denoising result by applying the PDRA to model~\eqref{task} with \(\alpha=0.05\)
is shown in Fig.~\ref{subfig:Camino:reg}.
We have used the parameters $\eta = 0.58$, $\lambda = 0.93$, and $\epsilon = 10^{-6}$ within the algorithm. The other three algorithms yield similar results. 

\begin{figure}\centering
	\begin{subfigure}{.49\textwidth}
		\includegraphics{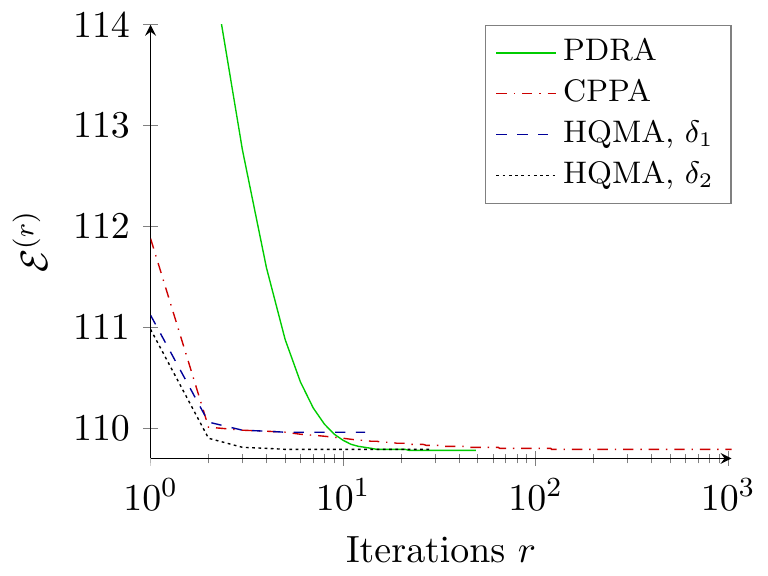}
		\caption[]{$\ell^2$-$\TV$ functional}
		\label{subfig:SPDplotl2TV}
	\end{subfigure}
	\begin{subfigure}{.49\textwidth}\centering
		\includegraphics{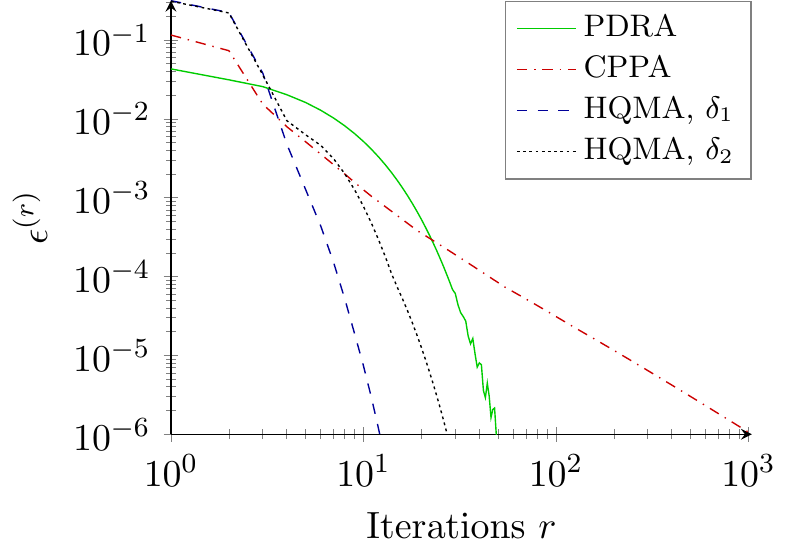}
		\caption[]{$\epsilon$ stopping criterion}
		\label{subfig:SPDploteps}
	\end{subfigure}
	\caption[]{For the development of both the \(\ell^2\)-\(\TV\) functional value in~\subref{subfig:SPDplotl2TV}
		and the stopping criterion \(\epsilon^{(r)} = \dist(x^{(r)},x^{(r-1)})\) in~\subref{subfig:SPDploteps}
		we compare the CPPA, DR and HQ minimization, with two choices of $\delta$. While the CPPA and the HQMA decrease the functional value a lot during the first iterations, 
		the PDRA yields the minimal value first and the HQ minimization can not reach the minimum due to the relaxation. 
		Regarding the $\epsilon^{(r)}$ the PDRA and HQMA behave similar, while the CPPA a slower decrease.}
	\label{fig:SPD:Iterations}
\end{figure}
Again we compare the development of  functional ${\mathcal E}^{(r)}$
and the error \(\epsilon^{(r)}\), where
we use \(\epsilon = 10^{-6}\) as stopping criterion. The results
are shown in Fig.~\ref{fig:SPD:Iterations}. 
Both values decay fast
in the first few iterations, so we plotted both with a logarithmic \(x\) axis. 
Comparing PDRA and CPPA, the results
are similar to the experiments on the Retina data.
Furthermore, the functional during the PDRA decreases below those of the HQMA, which is expected, because the HQMA is applied to the smoothed functional ${\mathcal E}_\delta$. 
The error \(\epsilon^{(r)}\) behaves quite similar in the PDRA and HQMA differing only by a factor.

The functional values, number of iterations and runtimes of the algorithms
are presented in Table~\ref{tab:SPD}.
While the HQMA  with \(\delta_1\) is the fastest both in number of iterations
and runtime, it is not able to reach the minimal functional value due to the smoothing. 
Even more, looking at a single iteration, the PDRA is faster than both HQMAs. 

Though there is no mathematical proof of converging to a minimizer, the PDRA
yields results, that still outperform the CPPA with respect to
iterations and runtime. It is compatible with the HQMA,
that minimizes a smoothed functional to gain performance.

\begin{table}\centering
	\begin{tabular}{lcccc}\toprule
		& PDRA & CPPA & HQMA, \(\delta_1\) & HQMA, \(\delta_2\)\\\midrule
		\(\ell^2\)-\(\TV\) value & $\mathbf{109.7844}$ & $109.7854$ & $109.9617$ & $109.7890$\\
		\# iterations & $49$ & $1042$ & $\mathbf{13}$ & $28$\\	
		runtime (sec.) & $139.2664$ & $643.2270$ & $\mathbf{113.0849}$ & $225.3906$\\\bottomrule
 	\end{tabular}
	\caption[]{Comparison of all four algorithms to minimize the \(\ell^2\)-\(\TV\) functional~\eqref{task} with respect to 
	the resulting minimal value, the number of iterations and the computational time.}
	\label{tab:SPD}
\end{table}

%------------------------------------------------------------------------------------------------
\subsection{Inpainting of Symmetric Positive Definite Matrices}\label{subsec:numerics:InpaintSPD}
%------------------------------------------------------------------------------------------------
\begin{figure}\centering
	\begin{subfigure}{.32\textwidth}
		\includegraphics[width=\textwidth]{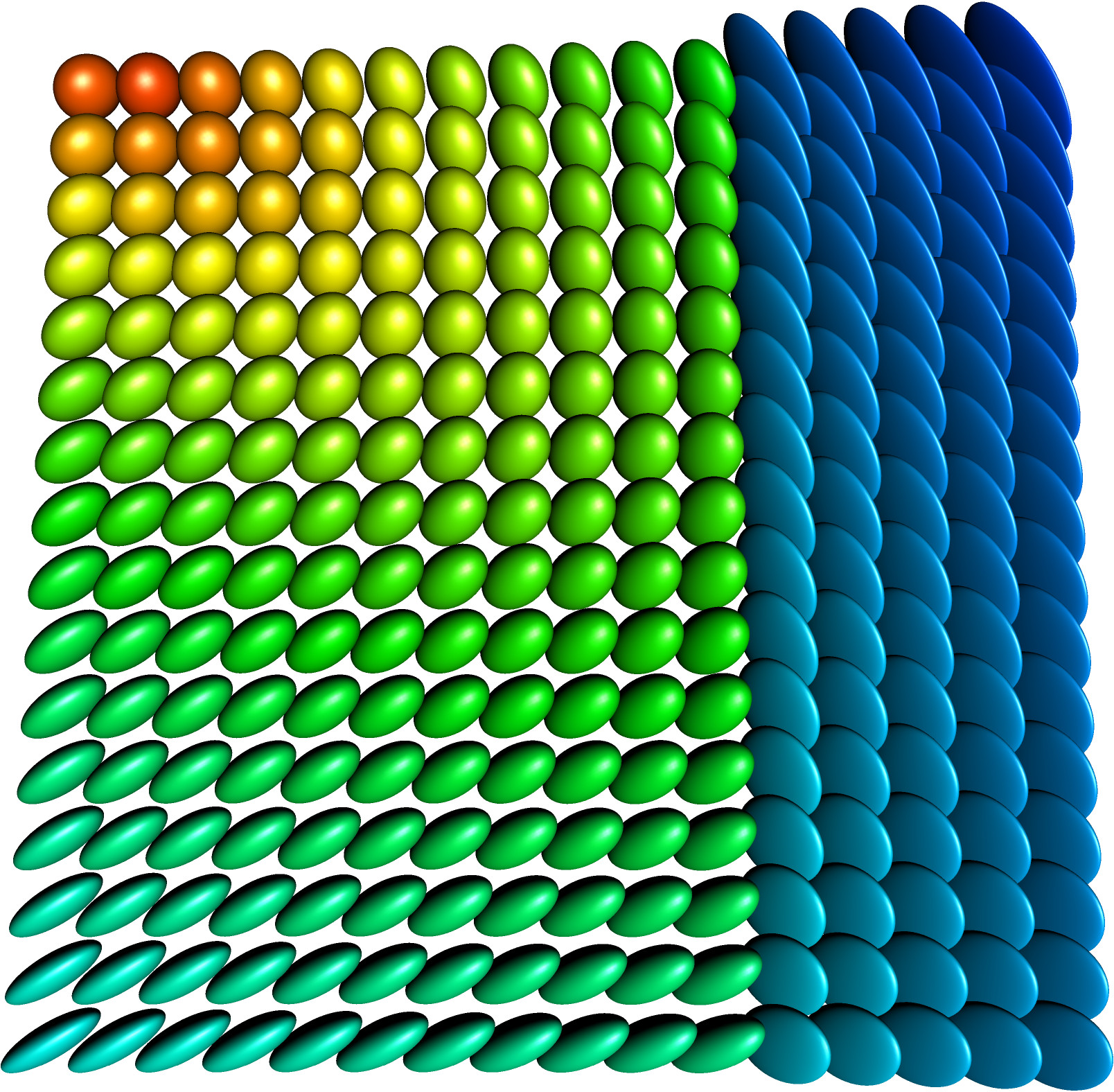}
		\caption[]{Original \(\SPD(3)\)-valued\\image}
		\label{fig:InpaintSPD:orig}
	\end{subfigure}
	\hfill
	\begin{subfigure}{.32\textwidth}
		\includegraphics[width=\textwidth]{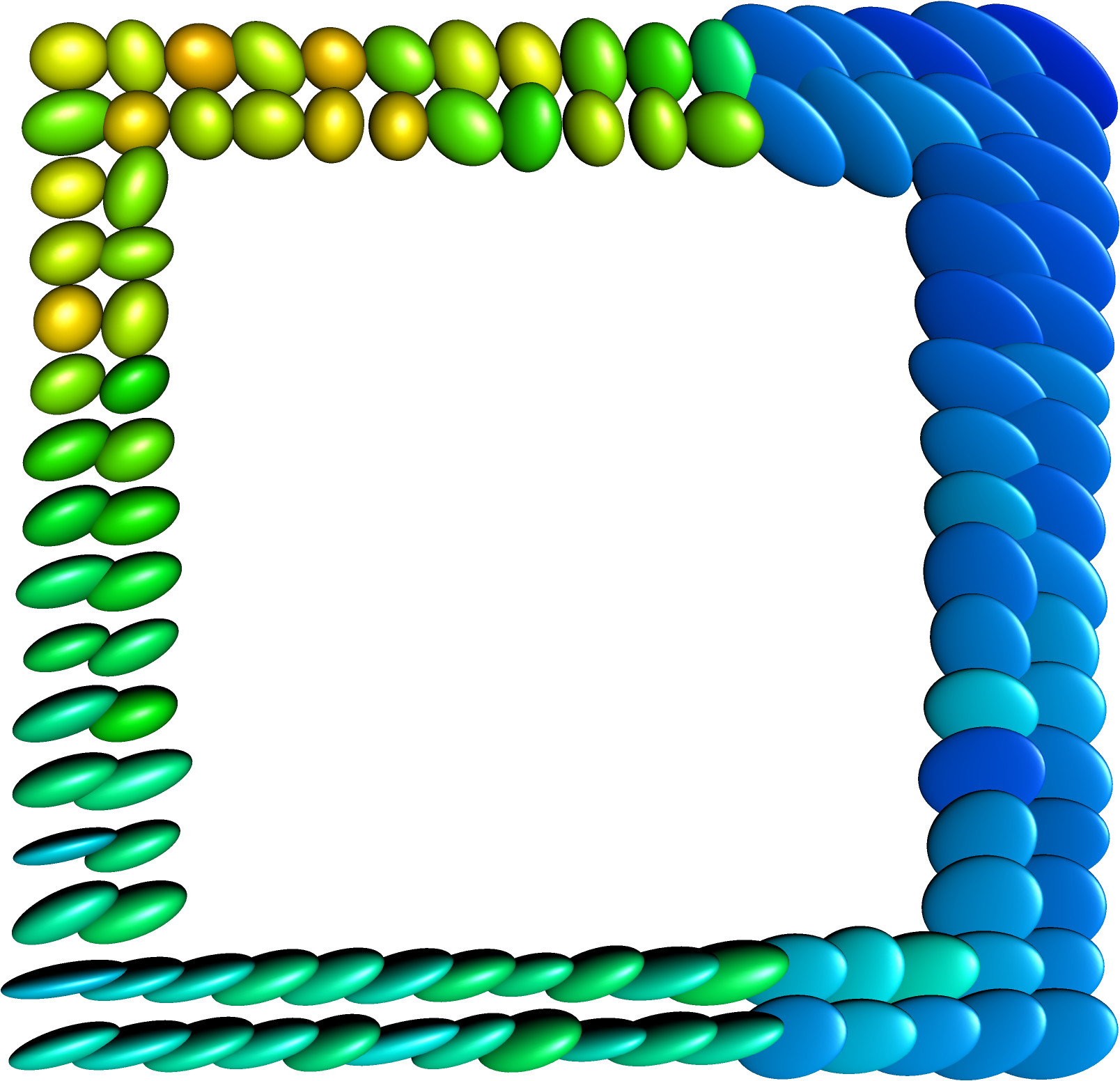}
		\caption[]{Noisy Data, \(\sigma_{\text{n}} = 0.01\),\\ on \(\mathcal V \subset \mathcal G\)}
		\label{fig:InpaintSPD:lossy}
	\end{subfigure}
	\hfill
	\begin{subfigure}{.32\textwidth}
		\includegraphics[width=\textwidth]{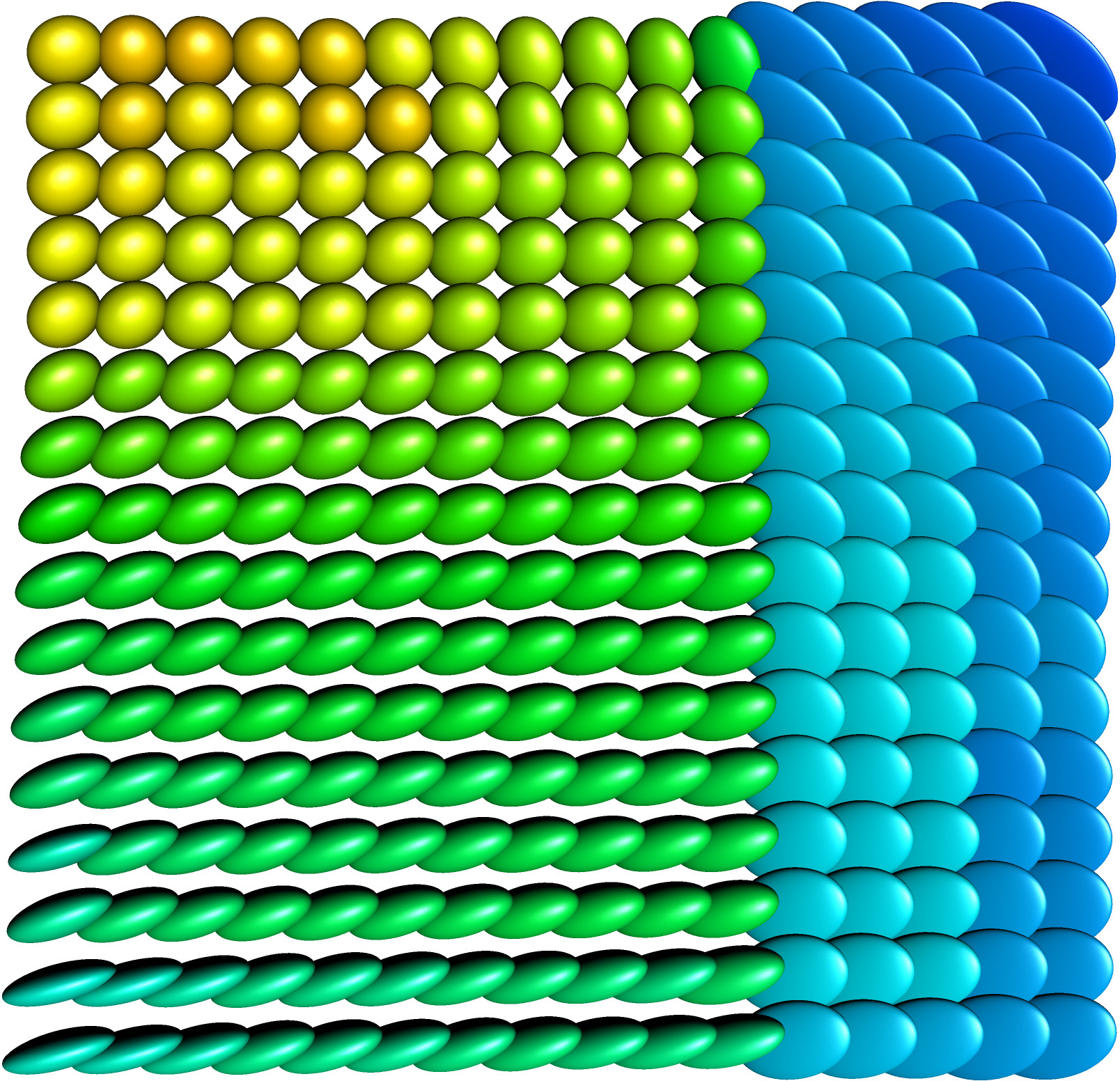}
		\caption[]{Inpainting \& denoising by\\model \eqref{task}, \(\alpha=0.01\)}
		\label{fig:InpaintSPD:DR}
	\end{subfigure}
	\caption[]{An artificial example of an \(\SPD(3)\)-valued image, where~\subref{fig:InpaintSPD:orig} an original image is~\subref{fig:InpaintSPD:lossy}
	obstructed by noise and loss of data. The reconstruction using the PDRA yields~\subref{fig:InpaintSPD:DR} a reconstruction that takes edges into account.}
	\label{fig:InpaintSPD}
\end{figure}

In many situations, the measured data is not only affected by noise, but 
several data items are lost. 
In the artificial experiment shown in
Fig.~\ref{fig:InpaintSPD}, the original image
\(f\) on $\mathcal G = \{1,\ldots,16\}^2$
is destroyed in the inner part, i.e.~only the indices
\(\mathcal V = \bigl\{(i,j) : \min\{i,16-i,j,16-j\}\leq 2\bigr\}\)
are kept, and the remaining data is disturbed by
Rician noise, \(\sigma_{\text{n}} = 0.01\), 
cf.~Fig.~\ref{fig:InpaintSPD:lossy}.
After initialization of the missing pixels by the nearest neighbor method, we
employ the PDRA to minimize \eqref{task} with \(\alpha = 0.1\), \(\eta = 3\) and \(\lambda = 0.95\). 
The algorithm stops after 117 iterations in 
\( 36.188 \) seconds with the stopping criterion 
\(\epsilon = 10^{-5}\). The result is shown in Fig. \ref{fig:InpaintSPD:DR}.
We obtain \(\mathcal E\bigl(u^{(117)}\bigr) = 5.7107\). 
The CPPA with \(\eta=12\)  stops after \(2210\) iterations
with the same stopping criterion as above in \(161\) seconds. 
The functional \(\mathcal E\) has a slightly higher value 5.7117.

In total, the PDRA requires far less iterations and a shorter computational time than the CPPA, even though one iteration of the PDRA takes four times as long as one iteration of the CPPA due to the computation of the Karcher mean.
Nevertheless, the functional value of the PDRA also beats the CPPA when using the same stopping criterion. So even if we do not have a proof of convergence, 
the PDRA even performs better than the CPPA in case of noisy and lossy data.
%
%-------------------------------------------------------------------------------
\section{Conclusions} \label{sec:conclusions}
%-------------------------------------------------------------------------------
%
We considered the restoration (denoising and inpainting) of images having values in symmetric Hadamard manifolds.
We proposed a model with an \(L_2^2\) data term, and anisotropic TV-like regularization term,
and examined the performance of a parallel Douglas Rachford algorithm for minimizing the corresponding functional.
Note that this carries over directly to an \(L_1\) data term.
Convergence can be proved for manifolds with constant non positive curvature.
Univariate Gaussian probability distributions or symmetric positive definite matrices with determinant 1
are typical examples of such manifolds.
Numerically, the algorithm works also well for other symmetric Hadamard manifolds as the 
symmetric positive definite matrices. 
However, having a look at the convergence proof, it is not
necessary that the reflection operators are nonexpansive for arbitrary values on the manifold, but instead
a fixed point is involved in the estimates.
Can the convergence  be proved under certain assumptions on the locality
of the data?
We are also interested in the question under which conditions the reflection operator a every proper convex lsc function on a
Hadamard manifold with constant curvature is nonexpansive.
We have seen that this is true for indicator functions of convex sets and various distance-like functions.
For the latter, we have proved that the reflection is nonexpansive on general symmetric Hadamard manifolds.
We are currently working on a software package for \textsc{Matlab} which includes
several algorithms to minimize ROF-like functionals involving first and second
order differences applied to manifold-valued images. The code will be made public
soon and the code is already available on request.

\paragraph{Acknowledgments}
We would like to thank M. Ba{\v{c}}{\'a}k for
bringing reference~\cite{LMWY2011} to our attention. Thanks to J.~Angulo for providing the data of the Retina experiment. 
We gratefully acknowledge funding by the DFG within the project STE 571/13-1 \& BE 5888/2-1.

%-----------------------------------------------------------------------------------------------------------------
\appendix

%---------------------------------------------------------------------------------
\section{Proof of Theorem~\ref{th:soln} }\label{app:theo}
%---------------------------------------------------------------------------------
\begin{proof}
	\begin{enumerate}[label={\arabic*.},leftmargin=0pt,itemindent=*]
\item Let $\hat x \in \HH^n$ be a a solution of~\eqref{eq:dr_split_had}. 
	Then we obtain by Theorem~\ref{th:subdiff_calculus} that
	\begin{align}
		0\in \partial\bigl(\varphi+\psi\bigr)(\hat x) = \partial{\varphi}(\hat x)+\partial{\psi}( \hat x).
	\end{align}
	This set inclusion can be split into two parts, namely there exists a point $y\in \mathcal{H}^n$ such that
	\begin{align}\label{eq:start}
		 \log_{\hat x} y \in \eta\partial \psi (\hat x) \; \text{ and } \; -\log_{\hat x} y\in\eta\partial \varphi(\hat x), 
	\end{align}
	i.e.,
	\begin{align} \label{eq:inclusions}
		0\in\eta\partial \psi(\hat x)-\log_{\hat x} y \; \text{ and } \; 0\in\eta\partial \varphi(\hat x) + \log_{\hat x} y.
	\end{align}
	By Theorem~\ref{th:subdiff_calculus} and since $\nabla d(y,\cdot)^2 (\hat x) = - 2 \log_{\hat x} y$, we have 
\begin{equation} \label{eq:prox_alter}
	\tilde x
	= \prox_{\eta \psi} (y)
	= \argmin_{z \in \HH^n}\bigl\{ \tfrac{1}{2} d(y,z)^2 + \eta \psi(z)\bigr\}
	\quad \Leftrightarrow \quad
	 0 \in -\log_{\tilde x} y  + \eta \partial \psi( \tilde x ).
\end{equation}
	Hence the first inclusion is equivalent to
	\begin{align}\label{eq:expression_g}
		\hat x  = \prox_{\eta \psi}(y).
	\end{align}
	This implies
	\begin{equation} \label{eq:log_exp}
		\exp_{\hat x}(-\log_{\hat x}y) = R_{\hat x} y = \mathcal{R}_{\eta \psi}(y).
	\end{equation}
	From the second inclusion in~\eqref{eq:inclusions} we obtain
	\begin{align}\label{eq:inc_to_ref}
		0\in \eta\partial \varphi({\hat x}) - \log_{\hat x}\bigl(\exp_{\hat x}(-\log_{\hat x}y)\bigr) = \eta\partial \varphi({\hat x}) - \log_{\hat x}\bigl( \mathcal{R}_{\eta \psi}(y) \bigr) .
	\end{align}
	Using again~\eqref{eq:prox_alter} this is equivalent to
	\begin{align} \label{eq:expression_f}
		{\hat x} = \prox_{\eta \varphi} \bigl(\mathcal{R}_{\eta{\psi}}(y)\bigr).
	\end{align}
	Now~\eqref{eq:log_exp} can be rewritten as
	\[
		y = \exp_{\hat x}\Bigl(-\log_{\hat x} \bigl( \mathcal{R}_{\eta \psi}(y) \bigr)\Bigr)
	\]
	and plugging in~\eqref{eq:expression_f} we get
	\begin{align}
	y &= \exp_{\prox_{\eta \varphi} {\displaystyle (}\mathcal{R}_{\eta{\psi}}(y){\displaystyle )}}
	\bigl(-\log_{\prox_{\eta \varphi} {\displaystyle (}\mathcal{R}_{\eta{\psi}}(y){\displaystyle )}}\mathcal{R}_{\eta \psi}(y)\bigr)\\
		   &=\mathcal{R}_{\eta \varphi}\mathcal{R}_{\eta \psi}(y).
		\end{align}
		Hence $y$ is a fixed point of $\mathcal{R}_{\eta \varphi}\mathcal{R}_{\eta \psi}$ which is related to ${\hat x}$ by~\eqref{eq:expression_g}.
\item Conversely, let $y$ be a fixed point of $\mathcal{R}_{\eta \varphi}\mathcal{R}_{\eta \psi}$. 
Expanding the reflection on $\eta\varphi$, we obtain
\begin{equation}\label{eq:expansion_r}
y = \exp_{\prox_{\eta \varphi} {\displaystyle (}\mathcal{R}_{\eta{\psi}}(y){\displaystyle )}}\bigl(-\log_{\prox_{\eta \varphi} 
{\displaystyle (}\mathcal{R}_{\eta{\psi}}(y){\displaystyle )}}\mathcal{R}_{\eta \psi}(y)\bigr).
\end{equation}
We set ${\hat x}\coloneqq \prox_{\eta\varphi}\bigl(\mathcal{R}_{\eta\psi}(y)\bigr)$.
Rewriting \eqref{eq:expansion_r} yields 
\begin{equation}
\exp_{\hat x}(-\log_{\hat x}y) = \exp_{\prox_{\eta\psi}(y)}(-\log_{\prox_{\eta\psi}(y)}y).
\end{equation}
From the uniqueness of geodesics we get 
\begin{equation}
\prox_{\eta\varphi}\bigl(\mathcal{R}_{\eta\psi}(y)\bigr)= \hat x = \prox_{\eta\psi}(y).
\end{equation}
By \eqref{eq:prox_alter} we conclude
\begin{equation}\label{eq:inc_psi}
0\in\eta\partial\psi({\hat x})-\log_{\hat x}y,
\end{equation}
and similar to \eqref{eq:inc_to_ref} we have
\begin{equation}\label{eq:inc_phi}
0\in\eta\partial\varphi({\hat x})+\log_{\hat x}y.
\end{equation}
Adding these inclusions we obtain 
\begin{equation}
0\in \partial{\varphi}({\hat x})+\partial{\psi}( {\hat x})\subseteq \partial\bigl(\varphi+\psi\bigr)({\hat x}),
\end{equation}
i.e., ${\hat x}$ is a solution of \eqref{eq:dr_split_had}.\qedhere
\end{enumerate}
\end{proof}

%---------------------------------------------------------------------------------
\section{Symmetric Positive Definite Matrices \texorpdfstring{\(\mathcal P(n)\)}{P(n)}}\label{app:spd}
%---------------------------------------------------------------------------------
The manifold $\bigl(\SPD(n),\langle\cdot,\cdot\rangle_{\SPD(n)}\bigr)$ of  symmetric positive definite  $n \times n$ matrices $\SPD(n)$ is given by 
\begin{equation}
\SPD(n)\coloneqq \bigl\{x\in\RR^{n\times n} : x = x^\text{T} \text{ and } \vect{a}^{\text{T}}x\vect{a} > 0 \; \text{ for all } \vect{a}\in\RR^n\bigr\}
\end{equation}
with the affine invariant metric
\begin{equation}
 \langle u, v\rangle_{\SPD(n)} \coloneqq \operatorname{Trace}\bigl(x^{-1}ux^{-1}v\bigr),\quad u,v\in T_x\SPD(n).
\end{equation}
We denote by $\Exp$ and $\Log$ the matrix exponential and logarithm defined by
\(
\Exp x \coloneqq \sum_{k=0}^\infty \frac{1}{k!} x^k
\)
and
\( 
\Log x \coloneqq \sum_{k=1}^\infty \frac{1}{k} (I-x)^k, \ \rho(I-x) < 1.
\)  
Note that another metric for $\SPD(n)$, the so-called Log-Euclidean metric, was proposed in \cite{AFPA2007,pennec2006riemannian}
which is not considered in this paper.

Further, we use the following functions, see, e.g.,~\cite{SH14}:
\begin{description}
	\item[Geodesic Distance.] The distance between two points  \(x,y\in\SPD(n)\) is defined as
\[d_{\SPD(n)} (x,y) \coloneqq \norm[\big]{\Log(x^{-\frac12}yx^{-\frac12} )}.
\]
	\item[Exponential Map.] The exponential map \(\exp_{x}\colon T_{x}\SPD(n)\to\SPD(n)\) at a point \(x\in\SPD(n)\) is defined as
	\[
\exp_{x} ( v) \coloneqq x^{\frac12} \Exp(x^{-\frac12} v x^{-\frac12}) x^{\frac12},\qquad  v\in T_{x}\SPD(n).
\]
	\item[Logarithmic Map.] The logarithmic map \(\log_{x}\colon\SPD(n)\to T_{x}\SPD(n)\) at a point \(x\in\SPD(n)\) is given by 
	\[
		\log_{x} (y) \coloneqq x^{\frac12} \Log(x^{-\frac12} y x^{-\frac12}) x^{\frac12},\qquad y\in\SPD(n).
\]
	
	\item[Unit Speed Geodesic.] The geodesic connecting \(x\) and \(y\) is given by
\[
\gamma_{\overset{\frown}{x,y} }(t)
	\coloneqq x^{\frac12} \Exp \bigl(t \Log (x^{-\frac12} y x^{-\frac12} ) \bigr)x^{\frac12},
	\quad t\in[0,1],
\]
with $\gamma_{\overset{\frown}{x,y} }(0)=x$ and $\gamma_{\overset{\frown}{x,y} }(1) = y$.
\end{description}
%
%------------------------------------------------------------------------------------------------
\section{Hyperbolic Space}\label{app:hn}
%------------------------------------------------------------------------------------------------
In this section, we recall equivalent models of the hyperbolic space $\Hn_{\mathrm{M}}$ 
which were used in our implementations.
We start with the hyperbolic manifold 
which is the basis of our computations 
and introduce the relevant manifold functions.
Then we consider other equivalent models,
namely the Poincar\'e ball, the Poincar\'e upper half-plane,
the manifold of univariate Gaussian probability measures, and the space of
symmetric positive definite matrices
with determinant 1. 
In particular, we determine the isometries from these spaces to the  hyperbolic manifold.

The \emph{hyperbolic manifold} \(\Hn^d_{\mathrm{M}}\) of dimension \(d\) can be embedded into the \(\mathbb R^{d+1}\) 
using the Minkowski inner product \(\langle x,y\rangle_{\mathrm{M}} \coloneqq -x_{d+1}y_{d+1} + \sum_{i=1}^d x_iy_i \). Then
\[
		\Hn_{\mathrm{M}}^d\coloneqq
		\Bigl\{
			x\in\mathbb R^{d+1}
			:
			 \langle x,x\rangle_{\mathrm{M}} = -x_{d+1}^2 + \sum_{i=1}^{d} x_i^2 = -1,\ x_{d+1}>0
		\Bigr\}
		\subseteq\mathbb R^{d+1},
	\]
together with the metric $g_\mathrm{M}\coloneqq \langle \cdot,\cdot\rangle_{\mathrm{M}}$ is a Riemannian manifold. It has curvature $-1$, see \cite{Lee97}.
In this paper we are interested in $d=2$.
By \(\sinh\) and \(\cosh\) we denote the sine and cosine hyperbolicus and their inverses by \(\arsinh\) and \(\arcosh\), respectively.
The following functions were used in our computations:
\begin{description}
	\item[Geodesic Distance.] The distance between two points \(x,y\in\Hn_{\mathrm{M}}^d\) is defined as
\[d_{\Hn_{\mathrm{M}}^d} (x,y) \coloneqq \arcosh(-\langle x,y\rangle_{\mathrm{M}}).
\]
	\item[Exponential Map.] The exponential map \(\exp_{x}\colon T_{x}\Hn^d_{\mathrm{M}}\to\Hn^d_{\mathrm{M}}\) at a point \(x\in\Hn^d_{\mathrm{M}}\) is 
	\[
\exp_{x} (v) \coloneqq \cosh(\sqrt{\langle v, v\rangle_{\mathrm{M}}})x + \sinh(\sqrt{\langle v, v\rangle_{\mathrm{M}}})\frac{ v}{\sqrt{\langle v, v\rangle_{\mathrm{M}}}},
\qquad v\in T_{x}\Hn^d_{\mathrm{M}}.
\]
	\item[Logarithmic Map.] The logarithmic map \(\log_{x}\colon\Hn^d_{\mathrm{M}}\to T_{x}\Hn^d_{\mathrm{M}}\) at a point \(x\in\Hn^d_{\mathrm{M}}\) is given by
	\[
		\log_{x} (y) \coloneqq
		\frac{
			\arcosh(-\langle x,y\rangle_{\mathrm{M}})
		}{
			\sqrt{\langle x,y  \rangle_{\mathrm{M}}^2-1}
		}\bigl(
			y + \langle x,y \rangle_{\mathrm{M}}x
		\bigr)
\]
	
	\item[Geodesic.] The geodesic connecting \(x\) and \(y\) reads
\begin{align*}
\gamma_{\overset{\frown}{x,y} }(t)
	&\coloneqq
	x
	\cosh\bigl(t\arcosh(-\langle x,y\rangle_{\mathrm{M}})\bigr)\\
	&\quad+ \frac{y+\langle x,y\rangle_{\mathrm{M}}}%
	{\sqrt{\langle x,y\rangle_{\mathrm{M}}^2-1}}
	\sinh\bigl(t\arcosh(-\langle x,y\rangle_{\mathrm{M}})\bigr),
	\qquad t\in[0,1],
\end{align*}
with $\gamma_{\overset{\frown}{x,y} }(0)=x$ and $\gamma_{\overset{\frown}{x,y} }(1) = y$.
\end{description}

Next we consider the other models together with the relevant bijections.
%-------------------------------------------------------------------------------

\paragraph{Poincar{\'e} ball $\big(\Hn_{\mathrm{B}}^{d},g_{\mathrm{B}} \big)$.}  Let $B^d\subset\RR^d$ be the unit ball with respect to the Euclidean distance.
Together with the metric 
		\begin{equation}
		g_{\mathrm{B}}(u,v) \coloneqq 4\frac{\langle u,v\rangle}{(1-\lVert x\rVert^2)^2} ,\quad u,v\in T_xB^d,
		\end{equation}
		it becomes a Riemannian manifold $\Hn_{\mathrm{B}}^{d}\subset\RR^d$ called the \emph{Poincar{\'e} unit ball}.
		It is equivalent to the hyperpolic manifold $\Hn_{\mathrm{M}}^{d}$ with isometry
		$\pi_1 \colon\Hn_{\mathrm{M}}^d\to \Hn_{\mathrm{B}}^{d}$ defined by
		\begin{equation}
		\pi_1(x)\coloneqq \frac{1}{1+x_{d+1}} \tilde x, 
		\quad 
		\pi_1^{-1}(y) = \frac{1}{1-\lVert y\rVert^2}\begin{pmatrix} 2y\\ 1+\lVert y\rVert^2 \end{pmatrix},
	\end{equation}
	where $x = (x_1, \ldots,x_d,x_{d+1})^\tT = (\tilde x^\tT,x_{d+1})^\tT$, see \cite[Proposition 3.5]{Lee97}.
\paragraph{ Poincar{\'e} half-space $\big(\Hn_{\mathrm{P}}^{d},g_{\mathrm{P}} \big)$.}
The upper half-space $\Hn_{\mathrm{P}}^{d}\coloneqq\{x\in\RR^d: x_d>0\}$. 
with the metric 
		\begin{equation}
		g_{\mathrm{P}} (u,v) \coloneqq \frac{\langle u,v\rangle}{x_d^2},\quad u,v\in T_{x}\Hn_{\mathrm{P}}^{d},
		\end{equation}
		is a Riemannian manifold known as \emph{Poincar{\'e} half-space}.
		It is equivalent to $\Hn_{\mathrm{B}}^{d}$ and thus to $\Hn_{\mathrm{M}}^{d}$ with isometry
		$\pi_2 \colon \Hn_{\mathrm{B}}^{d} \to\Hn_{\mathrm{P}}^d$ given by
	\begin{equation}
	\pi_2(x)= \frac{1}{\lVert \tilde x\rVert^2+(x_d-1)^2}\begin{pmatrix} 	2\tilde x\\	1-\lVert x\rVert^2 - x_d^2\end{pmatrix}, 
	\quad 
	\pi_2^{-1}(y) = \frac{1}{\lVert\tilde y\rVert^2 + (y_d+1)^2}
	\begin{pmatrix}
	2 \tilde y \\
	\lVert\tilde y\rVert^2 + y_d^2 - 1
	\end{pmatrix}\!,
	\end{equation}
	where $x = (x_1, \ldots,x_{d-1},x_{d})^\tT = (\tilde x^\tT,x_{d})^\tT$, see \cite[Proposition 3.5]{Lee97}. 
\paragraph{Univariate Gaussian probability measures $\big({\mathcal N}, g_{\mathrm{F}}\big)$.}
A distance measure for probability distributions with density function 
$\varphi(x,\theta)$ 
depending on the parameters $\theta=(\theta_1,\ldots,\theta_n)$  is given by the Fischer information matrix \cite[Chapter 11]{CT2006}
\begin{equation}
F (\theta)   = 
\Big( \int_{-\infty}^\infty  \varphi(x,\theta) \frac{\partial \ln \varphi(x,\theta)}{\partial \theta_i}\frac{\partial \ln \varphi(x,\theta)}{\partial \theta_j} \, dx \Big)_{i,j=1}^n.
\end{equation}
In the case of univariate Gaussian densities it reduces to
\begin{equation}
F(\mu,\sigma) =\begin{pmatrix}
\frac{1}{\sigma^2} & 0\\0 &\frac{2}{\sigma^2},
\end{pmatrix}.
\end{equation}
Then the Fischer metric is given by 
\begin{equation}
g_{\text{F}}(u,v) \coloneqq \frac{u_1v_1 + 2u_2v_2}{\sigma^2},\quad u,v\in T_{(\mu,\sigma)}\mathcal{N},
\end{equation}
i.e. $ds^2 = u^T F u = \frac{1}{\sigma^2} (u_1^2 + 2u_2^2)$.
Then, the isomorphism $\pi_3\colon {\mathcal N} \rightarrow \Hn_{\mathrm{P}}^2$ with
\[
	\pi_3 (\mu,\sigma) =  \bigl(\tfrac{\mu}{\sqrt{2}},\sigma\bigr), \qquad \pi_3^{-1} (x_1,x_2) = (\sqrt{2} x_1, x_2)
\]
is an isometry between $( {\mathcal N},g_\mathrm{F})$ and $( \Hn_{\mathrm{P}}^2, 2 g_\mathrm{P} )$, see, e.g.,~\cite{CSS2014}.
Thus, $( {\mathcal N},g_\mathrm{F})$ has curvature $-\frac12$.
We mention that $\pi_1^{-1}\circ \pi_2^{-1} \circ \pi_3$ is an isomorphism from ${\mathcal N}$ to our model hyperbolic manifold $\Hn_{\mathrm{M}}^2$.
\paragraph{Symmetric positive definite $2 \times 2$ matrices with determinant 1 $\big( {\mathcal P}_1(2), \langle\cdot,\cdot\rangle_{\SPD_1(2)} \big)$.} 
Following~\cite{CF2009}, an 
	isomorphism $\pi_3\colon\SPD_1(2)\rightarrow \Hn_{\mathrm{M}}^2$ is given by
	\begin{equation}
\pi_4 (a) \coloneqq 
\begin{pmatrix}
\frac{a_{11} - a_{22}}{2}\\
{\textstyle a_{12}}\\ 
\frac{a_{11} + a_{22}}{2}
\end{pmatrix},
\qquad
\pi_4^{-1}(x) = 
\begin{pmatrix}
	x_1+x_3 & x_2\\ 
	x_2 & x_3-x_1
\end{pmatrix},
\end{equation}
where	
\begin{equation}
a\coloneqq \begin{pmatrix}
a_{11} & a_{12}\\a_{12} & a_{22} \end{pmatrix} \quad \text{with} \quad \det a =  a_{11}a_{11} - a_{12}^2 = 1,
\end{equation}
and 
$x \coloneqq (x_1,x_2,x_3)^\tT$ with $x_3 >0$ and $x_1^2 + x_2^2 - x_3^2 = -1$.

The operator $\pi_3$ is an isometry between $\bigl(\SPD_1(2),\langle\cdot,\cdot\rangle_{\SPD_1(2)}\bigr)$
and $\bigl(\Hn_{\mathrm{M}}^2,2\langle\cdot,\cdot\rangle\bigr)$.
To verify this relation, we define the push-forward operator \(\phi_*\colon T_x\mathcal M\rightarrow T_{\phi(x)}\mathcal N\)
of a smooth mapping $\phi\colon\mathcal M\rightarrow \mathcal N$ between two manifolds $\mathcal M, \mathcal N$ as 
\begin{align}
\phi_* u (f) \coloneqq u(f\circ\phi),
\end{align}
where $f\in\C^{\infty}(\mathcal{N},\RR),\ u \in T_x\mathcal M$.
The push-forward $\phi_*$ of $\phi$ is also known as the differential
of $\phi$, for details see~\cite{Lee2000}.
\begin{lemma}
	Let $\bigl(\Hn_{\mathrm{M}}^2,2\langle \cdot,\cdot\rangle_{\mathrm{M}}\bigr)$, the model space ${\mathcal M}^2_{-\frac{1}{2}}$ embedded into the $\RR^3$, and\\
	$\bigl(\SPD_1(2),\langle\cdot,\cdot\rangle_\SPD\bigr)$, the symmetric positive definite $2\times 2$ matrices having determinant 1 with the affine invariant metric, be given.
	Then 
	\begin{align}
	\psi\colon\Hn_{\mathrm{M}}^2\rightarrow\SPD_1(2),
	\psi(x) = \pi_4^{-1}(x) = \begin{pmatrix}
	x_1+x_3 & x_2\\x_2 & x_3-x_1
	\end{pmatrix},
	\end{align}
	is an isometry.
	\begin{proof}
		We have to show that
		\begin{equation}
		2\langle u,u \rangle_{\mathrm{M}} = \bigl\langle\psi_*(u),\psi_*(u)\bigr\rangle_{\psi(x)},
		\end{equation}
		for all $u\in T_x\Hn_{\mathrm{M}}^2$, see $\cite{Lee97}$. The push-forward $\psi_*$ is given by
		\begin{equation}
		\psi_*(u) = \begin{pmatrix}
		u_1+u_3 & u_2\\u_2 & u_3-u_1
		\end{pmatrix}.
		\end{equation}
		Given $x\in\Hn_{\mathrm{M}}^2$ and $u\in T_x\Hn_{\mathrm{M}}^2$ we know 
		\begin{align}
		\langle x,x\rangle_{\mathrm{M}}&= x_1^2+x_2^2-x_3^2 = -1,\label{eq:hyper}\\
		\langle x,u\rangle_{\mathrm{M}}&= x_1u_1+x_2u_2-x_3u_3 =0.\label{eq:tangent}				
		\end{align}
		Further we have 
		\begin{equation}
		\bigl(\psi(x)\bigr)^{-1} = \begin{pmatrix}
		x_3-x_1 & -x_2\\-x_2 & x_1+x_3
		\end{pmatrix}.
		\end{equation}
		We start with
		\begin{align}
		\bigl\langle&\psi_*(u),\psi_*(u)\bigr\rangle_{\psi(x)} =\operatorname{trace}\Bigl(\bigl(\psi(x)\bigr)^{-1}\psi_*(u)\bigl(\psi(x)\bigr)^{-1}\psi_*(u)\Bigr) \\
		&=\operatorname{trace}\left(\Biggl(\begin{pmatrix}
		x_3-x_1 & -x_2\\-x_2 & x_1+x_3
		\end{pmatrix}\begin{pmatrix}
		u_1+u_3 & u_2\\u_2 & u_3-u_1
		\end{pmatrix}\Biggr)^2	\right)\\
		&=\operatorname{trace}\left(\begin{pmatrix}
		x_3u_1+x_3u_3-x_1u_1-x_1u_3-x_2u_2 & x_3u_2-x_1u_2-x_2u_3+x_2u_1\\
		x_3u_2+x_1u_2-x_2u_3-x_2u_1 & -x_2u_2+x_1u_3-x_1u_1+x_3u_3-x_3u_1
		\end{pmatrix}^2	\right)\\	
		&= 2\bigl( x_1^2u_1^2 - x_1^2u_2^2 + x_1^2u_3^2 - x_2^2u_1^2 + x_2^2u_2^2 + x_2^2u_3^2  + x_3^2u_1^2 + x_3^2u_2^2 + x_3^2u_3^2 \\&\qquad+ 4x_1x_2u_1u_2 - 4x_1x_3u_1u_3   - 4x_2x_3u_2u_3 \bigr)\\
		&= 2\bigl(u_1^2(x_1^2-x_2^2+x_3^2)+ u_2^2(-x_1^2+x_2^2+x_3^2)+ u_3^2(x_1^2+x_2^2+x_3^2)\\
		&\qquad+ 4x_1x_2u_1u_2 - 4x_1x_3u_1u_3   - 4x_2x_3u_2u_3\bigr).
		\end{align}
		Using \eqref{eq:hyper} we obtain
		\begin{align}		
		\bigl\langle\psi_*(u),\psi_*(u)\bigr\rangle_{\psi(x)} &= 2\bigl(u_1^2(2x_1^2+1)+ u_2^2(2x_2^2+1)+ u_3^2(2x_3^2-1)\\
		&\qquad+ 4x_1x_2u_1u_2 - 4x_1x_3u_1u_3   - 4x_2x_3u_2u_3\bigr)\\
		&= 2(u_1^2+u_2^2-u_3^2)\\
		&\qquad+2(u_1^2x_1^2+u_2^2x_2^2+u_3^2x_3^2+ 4x_1x_2u_1u_2 - 4x_1x_3u_1u_3   - 4x_2x_3u_2u_3)\\
		&= 2\langle u,u\rangle_{\mathrm{M}}+ 2(u_1x_1+u_2x_2-u_3x_3)^2.
		\end{align}
		With \eqref{eq:tangent} we obtain the assertion.
	\end{proof}
\end{lemma}
%-------------------------------------------------------------------------------
\bibliographystyle{abbrv}
\bibliography{DR-ref}
\end{document}